\documentclass[12pt]{amsart}

\usepackage{color}
\usepackage{amssymb}
\usepackage{amsmath}
\usepackage{mathtools}
\usepackage{mathrsfs}
\usepackage{graphicx}
\usepackage{tikz}
\usepackage{multicol}
\usepackage{url}
\usepackage{bm}
\usepackage{booktabs}
\usepackage{longtable,array,makecell}
\usepackage{enumitem}
\usepackage{microtype}
\allowdisplaybreaks

\numberwithin{equation}{section}

\theoremstyle{plain}
\newtheorem{thm}{Theorem}[section]
\newtheorem{lem}[thm]{Lemma}
\newtheorem{prop}[thm]{Proposition}
\newtheorem{cor}[thm]{Corollary}

\theoremstyle{definition}

\theoremstyle{remark}
\newtheorem{rem}[thm]{Remark}

\usepackage[colorlinks,citecolor=red,linkcolor=blue,hypertexnames=false]{hyperref}

\hypersetup{pdftitle={Intersections and Minkowski Sums of Four-Corner Cantor Dusts with the Unit Circle},pdfauthor={Jiangtao Li, Zhao Shen, Yufeng Wu}}

\newcommand{\R}{\mathbb{R}}
\newcommand{\N}{\mathbb{N}}
\newcommand{\Z}{\mathbb{Z}}
\newcommand{\Leb}{\mathcal{L}}
\newcommand{\cardc}{\mathfrak{c}}
\newcommand{\dimH}{\dim_{\mathrm H}}

\newcommand{\interior}{\operatorname{int}}
\newcommand{\diam}{\operatorname{diam}}
\newcommand{\conv}{\operatorname{conv}}

\begin{document}
\baselineskip 16pt

\title{Intersections and Minkowski Sums of Four-Corner Cantor Dusts with the Unit Circle}

\author{Jiangtao Li}
\address[]{School of Mathematics and Statistics\\HNP-LAMA\\Central South University\\Changsha, 410083, China}
\email{lijiangtao@csu.edu.cn}

\author{Zhao Shen}
\address[]{School of Mathematics and Statistics\\HNP-LAMA\\Central South University\\Changsha, 410083, China}
\email{sz1021@csu.edu.cn}

\author{Yufeng Wu}
\address[]{School of Mathematics and Statistics\\HNP-LAMA\\Central South University\\Changsha, 410083, China}
\email{yufengwu.wu@csu.edu.cn}



\begin{abstract}
For $0<\lambda<1/2$, let $K_{\lambda}$ be the attractor of the iterated function system $\{\lambda x, \lambda x+1-\lambda\}$, and put
$C_\lambda=K_{\lambda}\times K_{\lambda}$. We study the intersection
\[
E_\lambda=C_\lambda\cap S^1
\]
and the Minkowski sum
\[
A_{\lambda}=C_\lambda+S^1,
\]
where $S^1$ is the unit circle. For the intersection problem, we prove that
$E_\lambda$ has the cardinality of the continuum for $(\sqrt{3}-1)/2<\lambda<1/2$,
and
$\dimH E_\lambda>0$ for $\sqrt{2}-1<\lambda<1/2$. We also establish quantitative lower bounds for $\dimH E_\lambda$
for $\lambda$ near $1/2$; in particular,
$\dimH E_\lambda$ approaches $1$ as $\lambda\uparrow1/2$.

For the Minkowski sum problem, we prove that
\(A_{\lambda}\) has nonempty interior throughout the
previously open range \(1/4<\lambda<1/3\),
answering a question of Simon and Taylor~\cite{SimonTaylor2020}. Together with earlier results of Simon and Taylor, our theorem yields the complete
classification: $A_{\lambda}$ has nonempty interior in $\R^2$ if and only if $1/4<\lambda<1/2$. More generally, we prove that $C_\lambda+\Gamma$ has
nonempty interior for every regular $C^1$ closed curve
$\Gamma$ whenever $1/4<\lambda<1/2$.
\end{abstract}

\maketitle

\section{Introduction and main results}

Intersections of fractal sets with smooth curves and Minkowski sums of fractal sets with curves are two basic themes in geometric measure theory. Let \(F\subset\R^2\) be a planar fractal set and let
\(\Gamma\subset\R^2\) be a curve. The intersection problem concerns the cardinality and dimension of \(F\cap\Gamma\),
whereas the sum problem concerns the dimension, Lebesgue measure, and
interior of the Minkowski sum
\[F+\Gamma:=\{x+y:x\in F,\ y\in\Gamma\}.\]

The classical Marstrand
slicing theorem provides a useful benchmark for the intersection problem.
Throughout the paper, let $S^1:=\{(x,y)\in\R^2:x^2+y^2=1\}$ denote the unit circle.  For \(e\in S^1\) and \(t\in\R\),
let
\[
L_{e,t}:=\{u\in\R^2:u\cdot e=t\}.
\]
We write $\dimH$ for Hausdorff dimension and $\mathcal{H}^s$ for $s$-dimensional Hausdorff measure. Marstrand~\cite{Marstrand54} proved that if \(F\subset\R^2\) is Borel, then for every fixed \(e\in S^1\),
\[
\dimH(F\cap L_{e,t})
\leq \max\{\dimH F-1,0\}
\]
for Lebesgue almost every \(t\in\R\). Moreover, if $0<\mathcal{H}^s(F)<\infty$ for some $1<s<2$,
then, for \(\mathcal{H}^1\)-almost every \(e\in S^1\), the above upper bound is sharp:
there is a set of parameters \(t\) of positive Lebesgue measure
for which
\[
\dimH(F\cap L_{e,t})=s-1.\] Thus, in almost every direction,
a positive-measure family of parallel lines meets \(F\) in sets of the expected dimension
\(s-1\). See
\cite{Marstrand54,Mattila95} for more details on Marstrand's theorem.

Marstrand's theorem is an almost-everywhere statement and therefore does not determine the intersection with an individual prescribed line. Nor does it directly apply to a prescribed nonlinear curve, such as the unit circle. For such a fixed curve $\Gamma$,
even deciding whether \(F\cap\Gamma\) is finite or infinite can be
difficult.

For $0<\lambda<1/2$, let $K_{\lambda}$ be the attractor of the iterated function system $\{\lambda x, \lambda x+1-\lambda\}$. That is, $K_{\lambda}$ is the unique nonempty compact set satisfying
\[
K_{\lambda}=\lambda K_{\lambda}\cup(\lambda K_{\lambda}+1-\lambda).
\]
Write
\[
C_\lambda:=K_{\lambda}\times K_{\lambda}.
\]
We refer to $C_\lambda$ as the four-corner Cantor dust with contraction ratio $\lambda$. Both $K_\lambda$ and $C_\lambda$ are self-similar sets
satisfying the strong separation condition, generated respectively by
two and four similarities with contraction ratio $\lambda$. Hence 
\[\dimH K_{\lambda}=\frac{\log2}{\log(1/\lambda)}, \qquad \dimH C_{\lambda}=\frac{\log4}{\log(1/\lambda)}.\]
In particular, $\dimH C_{\lambda}>1$ precisely when $\lambda>1/4$. See \cite{Falconer2014,Hutchinson1981} for background on self-similar sets and iterated function systems.

The two sets studied in this paper are
\begin{equation*}
E_\lambda:=C_\lambda\cap S^1,
\qquad
A_{\lambda}:=C_\lambda+S^1.
\end{equation*}
The circle-intersection problem arose from Yu's study of the
distribution of missing-digit points near nondegenerate analytic manifolds~\cite{Yu2023}. Yu asked whether
\[
S^1\cap (K_{1/5}\times K_{1/5})
\]
is infinite. Jiang, Kong, Li, and Wang~\cite{JKLW2026} answered this
question negatively and initiated a systematic study of
\(E_\lambda\). They proved that $E_\lambda=\{(0,1),(1,0)\}$ for  $0<\lambda\leq2-\sqrt{3}$, that \(E_\lambda\) is nontrivial for
\(0.330384\leq\lambda<1/2\), and that it has continuum cardinality
for \(0.407493\leq\lambda<1/2\). Here $E_{\lambda}$ is called nontrivial if it contains a point other than $(0,1)$ and $(1,0)$. They also constructed a sequence
\(\lambda_n\searrow2-\sqrt{3}\) for which \(E_{\lambda_n}\) is
nontrivial, showing that the endpoint \(2-\sqrt{3}\) is sharp for the
uniform triviality statement. They further conjectured that \(E_\lambda\) is infinite for $
2-\sqrt{3}<\lambda<1/2$. This conjecture was recently disproved by Jiang and Xie
\cite[Theorems~1.1 and~1.2]{JiangXie2026}.
More precisely, for $n\geq 1$,  let $\lambda_n$ be the unique root in $(0,1/2)$ of
\[
\lambda^{2n+2}+\lambda^2-4\lambda+1=0.
\]
They explicitly determined $E_{\lambda_n}$, showing that it consists of exactly $6$ points. They also proved that, for every $n\geq1$,
there is a nonempty open interval $J_n\subset(\lambda_{n+1},\lambda_n)$
on which $E_\lambda=\{(0,1),(1,0)\}$.

For the middle-third Cantor set
\(K_{1/3}\), Du, Jiang, and Yao~\cite{DuJiangYao25R} proved that $E_{1/3}$
contains at least \(10,000,000\) distinct points \((x_i,y_i)\), each satisfying
\(x_i\notin\mathbb Q\). Subsequently, Jiang, Kong, Li, and Wang asked in
\cite[Question~5.1]{JKLW2026} whether $E_{1/3}$ is infinite
and, moreover, asked for its Hausdorff dimension. They conjectured
that
\[
\dimH E_{1/3}=\frac{2\log2}{\log3}-1=\dimH C_{1/3}-1,
\]
in agreement with the dimension heuristic suggested by Marstrand's slicing theorem.
The infinitude part of
\cite[Question~5.1]{JKLW2026} was recently answered by
Jiang and Xie~\cite[Theorem~1.3 and Corollary~1.4]{JiangXie2026}. Indeed, they established the
infinitude statement recorded in Theorem~\ref{thm:intersection}(i)
below, with the same threshold $\lambda_\infty$.
Their theorem applies more generally to graphs of functions that are
$C^2$ near $0$ and satisfy $f(0)=1$, $f'(0)=0$, and $f''(0)=-1$. After submitting the first version of this paper, we learned of the
earlier manuscript of Jiang and Xie, which had been submitted in
August 2026. We obtained the circle infinitude result independently,
but acknowledge their priority for
Theorem~\ref{thm:intersection}(i).
We retain our proof for completeness and to keep the exposition
self-contained.

In this paper, we establish   new results on the intersection problem.  Specifically, we improve the parameter range for continuum
cardinality, prove positive Hausdorff dimension,
and obtain quantitative dimension estimates near $\lambda=1/2$.
The following theorem collects these results in parts~\textup{(ii)--(iv)},
together with the infinitude statement of Jiang and Xie in
part~\textup{(i)}.
 
\begin{thm}\label{thm:intersection}
Let $0<\lambda<1/2$. Let \(\lambda_{\infty}\approx0.305854\) be the unique root in $(1/4,1/2)$ of $2x^3-3x^2+4x-1=0$, and set 
\begin{equation*}
\lambda_*:=\frac{\sqrt{3}-1}{2},
\qquad
\lambda_{\mathrm H}:=\sqrt{2}-1.
\end{equation*}
Then the following statements hold.
\begin{enumerate}[label=\textup{(\roman*)}]
\item If $\lambda_{\infty}<\lambda<1/2$, then $E_\lambda$ is infinite.  In fact, it contains a sequence of distinct points converging to $(0,1)$ and a symmetric sequence converging to $(1,0)$.
\item If $\lambda_*<\lambda<1/2$, then $E_\lambda$ has the cardinality of the continuum.
\item If $\lambda_{\mathrm H}<\lambda<1/2$, then $\dimH E_\lambda>0$.
\item 
If $1/2-3\times10^{-9}\leq\lambda<1/2$, then
\begin{equation*}
\dimH E_\lambda\geq1-2102\frac{1-2\lambda}{\lambda}.
\end{equation*}
Consequently, $\lim_{\lambda\uparrow1/2}\dimH E_\lambda=1$.
\end{enumerate}
\end{thm}

A sharper two-sided first-order estimate for
\(1-\dimH E_\lambda\) as \(\lambda\uparrow1/2\) is given in
Theorem~\ref{thm:dim-asy}.

We next turn to the Minkowski sum problem. Minkowski sums of planar sets and curves were studied
systematically by Simon and Taylor
\cite{SimonTaylor2020,SimonTaylor2022}. Here and below, \(\Leb^2\) denotes the two-dimensional Lebesgue measure. Simon and Taylor proved that 
\[
\dimH A_{\lambda}
=
1+\frac{\log4}{\log(1/\lambda)}<2
\qquad
\left(0<\lambda<\frac{1}{4}\right),
\]
whereas
\[
\dimH A_{1/4}=2
\quad\text{and}\quad
\Leb^2(A_{1/4})=0.
\]
They also proved that \(\Leb^2(A_{\lambda})>0\) for
\(\lambda>1/4\), and that \(A_{\lambda}\) has nonempty
interior for \(\lambda\geq1/3\). They left open whether \(A_{\lambda}\) has nonempty interior in the
remaining range  \(1/4<\lambda<1/3\); see \cite[Remark~2.8]{SimonTaylor2020}. Our second main result closes this gap.

\begin{thm}\label{thm:sum}
If $1/4<\lambda<1/2$, then $A_{\lambda}$ has nonempty interior. 
\end{thm}

Combining Theorem~\ref{thm:sum} with the results of Simon and Taylor~\cite{SimonTaylor2020,SimonTaylor2022} yields the following classification. 

\begin{cor}\label{cor:classification}
Let $0<\lambda<1/2$. Then 
$A_{\lambda}$ has nonempty interior if and only if $\lambda>1/4$. 
\end{cor}

The proof of Theorem~\ref{thm:sum} is local.
When $\lambda>1/4$, we choose a horizontal parameter and a
sufficiently small Cantor cylinder so that the local slope
conditions of Lemma~\ref{lem:nonlinear-fill} are satisfied.
This nonlinear interval-filling lemma then shows that the
corresponding image of the Cantor product is a nondegenerate
interval. By continuity, varying the horizontal parameter
over a small open interval yields a common vertical interval
and hence an open rectangle contained in $A_{\lambda}$. The same argument also shows that
$C_\lambda+\Gamma$ has nonempty interior for every regular
$C^1$ closed curve $\Gamma$ whenever $1/4<\lambda<1/2$;
see Corollary~\ref{cor:closed-curve-sum}.

The paper is organized as follows. Section~\ref{sec:prelim} fixes notation and recalls some basic facts.  Section~\ref{sec:infinite} proves the infinitude result for $E_\lambda$. Section~\ref{sec:continuum} establishes the continuum-cardinality
assertion, and Section~\ref{sec:dimension} gives the dimension estimates. Section~\ref{sec:sum} proves  Theorem~\ref{thm:sum}. Finally, in Section~\ref{sec:open-problems}, we discuss the remaining
gaps between the known parameter ranges and formulate several open
problems.

\section{Preliminaries}\label{sec:prelim}

Throughout, $\log$ denotes the natural logarithm. Let $\N:=\{1,2,\ldots\}$ be the set of positive integers and $\N_0:=\N\cup\{0\}$. 
For a set $A\subset\R^d$, its cardinality, convex hull, diameter, and interior are denoted by $\#A$, $\conv A$,  $\diam A$, and  $\interior A$, respectively.  For an interval $I\subset \R$, $|I|$ denotes its length. The continuum cardinality is denoted by $\cardc=2^{\aleph_0}$.

For $n\in\N_0$, let $\{0,1\}^n$ be the set of all binary words of length $n$, and let 
\[
\{0,1\}^*
:=
\bigcup_{n\in\N_0}\{0,1\}^n.
\]
We adopt the convention that $\varnothing$ is the empty word and $\{0,1\}^0=\{\varnothing\}$. For \(w\in\{0,1\}^*\), let \(|w|\) denote its length.
Concatenation of two finite words \(u,v\) is denoted by \(uv\).
For \(n\in\N\), the words consisting of \(n\) zeros and \(n\)
ones are denoted by \(0^n\) and \(1^n\), respectively.

For $0<\lambda<1/2$, set
\[
f_0(x):=\lambda x,
	\qquad
	f_1(x):=\lambda x+1-\lambda.
\]
For \(n\in\N\) and
\(w=w_1\cdots w_n\in\{0,1\}^n\),  define
\[
f_w:=f_{w_1}\circ\cdots\circ f_{w_n}, \quad I_w:=f_w([0,1]),
\quad
K_w:=f_w(K_{\lambda}).\]
Note that $f_w(x)=\lambda^nx+a_w$ and thus $I_w=[a_w,a_w+\lambda^n]$, where 
\begin{equation*}
	a_w=(1-\lambda)\sum_{j=1}^n w_j\lambda^{j-1}.
\end{equation*}
For the empty word, set
$f_\varnothing:=\operatorname{id}$, the identity map, $I_\varnothing:=[0,1]$ and $K_\varnothing:=K_{\lambda}$. With this convention, for all
\(u,v\in\{0,1\}^*\), 
\[f_{uv}=f_u\circ f_v, \quad 
I_{uv}=f_u(I_v)\subset I_u, \quad  K_{uv}=f_u(K_v)\subset K_u.\]
For \(w\in\{0,1\}^*\) with   \(|w|=n\),   we call \(I_w\) and \(K_w\) a level-\(n\) cylinder interval and a level-\(n\) cylinder set,
respectively. A point that is an endpoint of some cylinder interval is called a cylinder endpoint.

Since $\lambda<1/2$, if
\(u,v\in\{0,1\}^n\) and \(u\neq v\), then $I_u\cap I_v=\varnothing$ and $K_u\cap K_v=\varnothing$. Consequently, $K_{w}=K_{\lambda}\cap I_w$ for $w\in \{0,1\}^*$.

On $[0,1]$, define
\begin{equation*}
T(x):=\sqrt{1-x^2}.
\end{equation*}
The map $T$ is continuous and strictly decreasing,  and it satisfies $T(T(x))=x$ for all $x\in [0,1]$. On $(0,1)$, put
\[
\psi(x):=|T'(x)|
=\frac{x}{\sqrt{1-x^2}}.
\]
Write
\begin{equation*}
X_\lambda:=K_{\lambda}\cap T(K_{\lambda}).
\end{equation*}
Then the intersection $E_{\lambda}$ can be written as the graph of $T$ over $X_{\lambda}$: 
\[E_\lambda=\{(x,T(x)):x\in X_\lambda\}.\]

\section{A sufficient threshold for infinitely many intersection points}\label{sec:infinite}

In this section we prove Theorem~\ref{thm:intersection}(i). We use a nonlinear interval-filling lemma to construct infinitely
many points of $E_\lambda$ accumulating at $(0,1)$ and $(1,0)$.

The interval-filling lemma below follows, by an affine change of
variables, from Jiang's result
\cite[Corollary~1.8]{Jiang2022}. We include a direct proof for completeness.

\begin{lem}\label{lem:filling}
Let $1/4\leq\lambda<1/2$. Set
$Q=[p,p+\ell]\times[q,q+\ell]$, where $p,q\in \R$  and $\ell>0$. Let $F=F(u,v)$ be a $C^1$ map from an open neighborhood of $Q$ to $\R$.  Suppose that throughout $Q$, the partial derivatives of $F$ satisfy 
\begin{equation}\label{eq:Fuv}
 F_u<0<F_v,
 \qquad
 \frac{1-2\lambda}{\lambda}
 \leq -\frac{F_u}{F_v}
 \leq\frac{1}{1-2\lambda}.
\end{equation}
Then
\[
 F\bigl((p+\ell K_{\lambda})\times(q+\ell K_{\lambda})\bigr)=F(Q).
\]
\end{lem}

\begin{proof}
Write $d_0=0$ and $d_1=1-\lambda$, and let
\[
 Q_{ij}=[p+d_i\ell,p+(d_i+\lambda)\ell]
 \times[q+d_j\ell,q+(d_j+\lambda)\ell],
 \qquad i,j\in\{0,1\}.
\]
Since $F_u<0<F_v$, the image of every subsquare
$[r,r+s]\times[t,t+s]\subset Q$ is the interval
\begin{equation*}
 F([r,r+s]\times[t,t+s])=[F(r+s,t),F(r,t+s)].
\end{equation*}
	Consequently,
\[
\begin{aligned}
	F(Q_{10})
	&=
	[F(p+\ell,q),
	F(p+(1-\lambda)\ell,q+\lambda\ell)],\\
	F(Q_{11})
	&=
	[F(p+\ell,q+(1-\lambda)\ell),
	F(p+(1-\lambda)\ell,q+\ell)],\\
	F(Q_{00})
	&=
	[F(p+\lambda\ell,q),
	F(p,q+\lambda\ell)],\\
	F(Q_{01})
	&=
	[F(p+\lambda\ell,q+(1-\lambda)\ell),
	F(p,q+\ell)].
\end{aligned}
\]

We compare the four intervals in the order $F(Q_{10})$, $F(Q_{11})$,  $F(Q_{00})$,  $F(Q_{01})$.
The assumption \eqref{eq:Fuv} implies, throughout
\(Q\),
\begin{equation}\label{eq:FuvInq}
\lambda F_u+(1-2\lambda)F_v\leq 0,\qquad
F_u+(1-2\lambda)F_v\leq 0,\qquad
(1-2\lambda)F_u+F_v\geq 0.
\end{equation}
Integrating the first inequality along segments parallel to
\((\lambda,1-2\lambda)\) gives
\[
F(p+\ell,q+(1-\lambda)\ell)
\leq
F(p+(1-\lambda)\ell,q+\lambda\ell)
\]
and
\[
F(p+\lambda\ell,q+(1-\lambda)\ell)
\leq
F(p,q+\lambda\ell).
\]
Hence
\[
F(Q_{10})\cap F(Q_{11})\neq\varnothing,
\qquad
F(Q_{00})\cap F(Q_{01})\neq\varnothing.
\]

For the middle two intervals, integrating the second inequality in \eqref{eq:FuvInq}
along the segment with displacement \(\ell(1,1-2\lambda)\) gives
\[
F(p+\ell,q+(1-\lambda)\ell)
\leq
F(p,q+\lambda\ell),
\]
while integrating the third inequality in \eqref{eq:FuvInq} along the segment with
displacement \(\ell(1-2\lambda,1)\) gives
\[
F(p+\lambda\ell,q)
\leq
F(p+(1-\lambda)\ell,q+\ell).
\]
Thus
\[
F(Q_{11})\cap F(Q_{00})\neq\varnothing.
\]
The four intervals therefore form an overlapping chain. Since their
union contains the endpoints \(F(p+\ell,q)\) and \(F(p,q+\ell)\)
of \(F(Q)\), it follows that
\begin{equation}\label{eq:FQFQij}
 F(Q)=\bigcup_{i,j\in\{0,1\}}F(Q_{ij}).
\end{equation}

The derivative assumptions in \eqref{eq:Fuv} remain valid on every subsquare in $Q$.  Let $\mathcal{K}_n$ be the level-$n$ approximation of $K_{\lambda}$, i.e., $\mathcal{K}_n:=\bigcup_{|w|=n}I_w$ and $\mathcal{K}_0:=[0,1]$.  Set
\[
 \mathcal{E}_n=(p+\ell \mathcal{K}_n)\times(q+\ell \mathcal{K}_n).
\]
Repeated application of \eqref{eq:FQFQij} gives
$F(Q)=F(\mathcal{E}_n)$ for every $n$.  Fix \(t\in F(Q)\).
Then
$\mathcal{E}_n\cap F^{-1}(\{t\})$ ($n\geq0$)
is a nested sequence of nonempty compact sets. Hence its
intersection is nonempty. Since
\[
\bigcap_{n\geq0}\mathcal{E}_n
=(p+\ell K_{\lambda})\times(q+\ell K_{\lambda}),
\]
we obtain
\[
t\in
F\bigl((p+\ell K_{\lambda})\times(q+\ell K_{\lambda})\bigr).
\]
The reverse inclusion is immediate. This proves the lemma.
\end{proof}

The next lemma gives a sufficient condition for the intersection $E_{\lambda}$ to be infinite. 

\begin{lem}\label{lem:infinity-criterion}
Let $1/4<\lambda<1/2$. 
Let $k\in \N$ and  $w,z\in\{0,1\}^k$. Write $I_w=[L_w,U_w]$ and $I_z=[L_z,U_z]$. If
\begin{equation}\label{eq:LwUw}
 L_w>1-2\lambda,
 \quad (1-2\lambda)U_w<\lambda,
 \quad U_w^2\geq2\lambda L_z,
 \quad L_w^2<2\lambda U_z,
\end{equation}
then $E_\lambda$ contains a sequence of pairwise distinct points
converging to $(0,1)$. In particular, $E_\lambda$ is infinite.
\end{lem}

\begin{proof}
Put $u(\xi):=L_w+\lambda^k\xi$ and
$v(\eta):=L_z+\lambda^k\eta$.
Then $u(K_\lambda)=K_w$ and $v(K_\lambda)=K_z$.
For $\varepsilon>0$, define
\[
H_{\varepsilon}(\xi,\eta)
:=\lambda v(\eta)-\frac{u(\xi)^2}{2}-\frac{\varepsilon v(\eta)^2}{2}.
\] The auxiliary function $H_{\varepsilon}$ is chosen so that its zeros
produce points on the unit circle. More precisely, if
$\varepsilon=\lambda^{2n+2}$
and $
H_{\varepsilon}(\xi,\eta)=0$,
then, writing \(u=u(\xi)\) and \(v=v(\eta)\), one has
\[
(\lambda^n u)^2+
(1-\lambda^{2n+1}v)^2=1.
\]
We therefore first solve \(H_{\varepsilon}=0\) for every sufficiently small \(\varepsilon>0\).

The fourth condition in \eqref{eq:LwUw} implies $U_z>0$.  Define
\begin{equation}\label{eq:epsilon0}
 \varepsilon_0:=\frac{1}{2}\min\left\{
 \frac{\lambda-(1-2\lambda)U_w}{U_z},
 \frac{2\lambda U_z-L_w^2}{U_z^2}
 \right\}>0.
\end{equation}
Fix $0<\varepsilon\leq\varepsilon_0$. Throughout $[0,1]^2$, $L_w\leq u\leq U_w$, $0\leq v\leq U_z$, and  
\[\lambda-\varepsilon v\geq \lambda-\varepsilon U_z\geq \frac{\lambda+(1-2\lambda)U_w}{2}>(1-2\lambda) U_w>0.\]
Hence $(H_{\varepsilon})_{\xi}=-\lambda^ku<0$ and 
$(H_{\varepsilon})_{\eta}=\lambda^k(\lambda-\varepsilon v)>0$.
Moreover,
\[
 \frac{(1-2\lambda)}{\lambda}<\frac{L_w}{\lambda}\leq -\frac{(H_{\varepsilon})_{\xi}}{(H_{\varepsilon})_{\eta}}
 =\frac{u}{\lambda-\varepsilon v}
 \leq\frac{U_w}{\lambda-\varepsilon U_z}<\frac{1}{1-2\lambda}.
\]
Thus Lemma~\ref{lem:filling} applies to $H_{\varepsilon}$.

The function $H_{\varepsilon}$ decreases in $\xi$ and increases in $\eta$.  By the third condition in \eqref{eq:LwUw},
\[
 H_{\varepsilon}(1,0)
 =\lambda L_z-\frac{U_w^2}{2}-\frac{\varepsilon L_z^2}{2}<0.
\]
Indeed, this clearly holds when $U_w^2>2\lambda L_z$. 
If $U_w^2=2\lambda L_z$, then
$L_z=U_w^2/(2\lambda)>0$ since
$U_w\geq L_w>1-2\lambda>0$, so the last term still gives strict negativity. Also, by \eqref{eq:epsilon0},
\[
 H_{\varepsilon}(0,1)
 =\lambda U_z-\frac{L_w^2}{2}-\frac{\varepsilon U_z^2}{2}>0.
\]
Lemma~\ref{lem:filling} therefore gives 
\[H_{\varepsilon}(K_{\lambda}\times K_{\lambda})=H_{\varepsilon}([0,1]^2)=[H_{\varepsilon}(1,0), H_{\varepsilon}(0,1)]\ni0.\]
Hence there exist \(\xi,\eta\in K_{\lambda}\) such that
\(H_{\varepsilon}(\xi,\eta)=0\).

Choose \(n_0\in\N\) such that $\lambda^{2n+2}\leq\varepsilon_0$ for all $n\geq n_0$.
For each \(n\geq n_0\), choose \((\xi_n,\eta_n)\in K_{\lambda}\times K_{\lambda}\) satisfying
$H_{\lambda^{2n+2}}(\xi_n,\eta_n)=0$.
Set
\[
u_n:=u(\xi_n),
\qquad
v_n:=v(\eta_n),
\qquad
x_n:=\lambda^n u_n,
\qquad
y_n:=1-\lambda^{2n+1}v_n.
\]
Since $u_n,v_n\in K_{\lambda}$, $\lambda^m K_{\lambda}\subset K_{\lambda}$ for all $m\in\N$, and $1-K_{\lambda}=K_{\lambda}$,  
we have $x_n,y_n\in K_{\lambda}$. Furthermore,
\[
 x_n^2+y_n^2
 =1+\lambda^{2n}
 \bigl(u_n^2-2\lambda v_n+\lambda^{2n+2}v_n^2\bigr)=1.
\]
Thus 
$(x_n,y_n)\in E_\lambda$ for all $n\geq n_0$. Since
\[
0<\lambda^nL_w\leq x_n\leq\lambda^nU_w\longrightarrow0
\]
and
\[
0\leq1-y_n
=\lambda^{2n+1}v_n
\leq\lambda^{2n+1}U_z\longrightarrow0,
\]
we have $(x_n,y_n)\to(0,1)$.
Since $x_n>0$ and $x_n\to0$, we can pass to a subsequence
whose first coordinates are strictly decreasing.
This gives a sequence of pairwise distinct points of
$E_\lambda$ converging to $(0,1)$.
In particular, $E_\lambda$ is infinite.
\end{proof}

\begin{proof}[Proof of Theorem~\ref{thm:intersection}(i)]
	Apply Lemma~\ref{lem:infinity-criterion} with $w=z=10$.
	The corresponding cylinder interval is $I_{10}=[L,R]$, where 
	\[L=1-\lambda, \qquad R=1-\lambda+\lambda^2.\]
	The first condition in \eqref{eq:LwUw} holds because $
	L-(1-2\lambda)=\lambda>0$.
	Moreover,
	\[
	\lambda-(1-2\lambda)R
	=
	2\lambda R-L^2
	=
	P(\lambda),
	\]
where $P(x):=2x^3-3x^2+4x-1$.
Since
	$P'(x)
	=
	6\left(x-\frac{1}{2}\right)^2+\frac{5}{2}>0$, $P$ is strictly increasing. Furthermore, $
	P\left(\frac{1}{4}\right)<0<P\left(\frac{1}{2}\right)$. Hence $P$ has a unique root $\lambda_{\infty}$ in $(1/4,1/2)$, and $P(\lambda)>0$ for $\lambda_\infty<\lambda<1/2$. Therefore,
	the second and fourth conditions in \eqref{eq:LwUw} hold
	when $\lambda>\lambda_{\infty}$.
	
	Finally, since $
	R^2-2\lambda L
	=
	(1-2\lambda)^2
	+
	\lambda^2(1-\lambda)^2
	>0$, the third condition also holds. Thus all four
	conditions in \eqref{eq:LwUw} are
	satisfied. 
	Lemma~\ref{lem:infinity-criterion} gives a sequence of pairwise
	distinct points of $E_\lambda$ converging to $(0,1)$.
	By invariance of $E_{\lambda}$ under coordinate interchange, there is also
	such a sequence converging to $(1,0)$.
\end{proof}

\section{Continuum cardinality of the circle intersection}
\label{sec:continuum}

In this section we prove Theorem~\ref{thm:intersection}(ii).
We first establish a double-covering criterion for pairs of
cylinder intervals, possibly belonging to different levels. We
then verify this criterion throughout the range
\(\lambda_*<\lambda<1/2\). Finally, we show that
\(\lambda_*\) is the optimal lower endpoint for the present
criterion.

\subsection{A double-covering criterion for cylinder pairs}

Let $0<\lambda<1/2$ be fixed.
Let 
\[I=[a,a+L], \qquad J=[b,b+M]\]
be two cylinder intervals, where $L=\lambda^n$ and $M=\lambda^m$ for possibly distinct $m,n\in\N$. Their first-generation children are
\begin{align*}
I_0&=[a,a+\lambda L],&
I_1&=[a+(1-\lambda)L,a+L],\\
J_0&=[b,b+\lambda M],&
J_1&=[b+(1-\lambda)M,b+M].
\end{align*}
Define $g(x,y):=x^2+y^2$. For $i,j\in\{0,1\}$, write $G_{ij}:=g(I_i\times J_j)=[\ell_{ij},u_{ij}]$. 
Explicitly,
\begin{align*}
	\ell_{00}&=a^2+b^2,
	&u_{00}&=(a+\lambda L)^2+(b+\lambda M)^2,\\
	\ell_{10}&=(a+(1-\lambda)L)^2+b^2,
	&u_{10}&=(a+L)^2+(b+\lambda M)^2,\\
	\ell_{01}&=a^2+(b+(1-\lambda)M)^2,
	&u_{01}&=(a+\lambda L)^2+(b+M)^2,\\
	\ell_{11}&=(a+(1-\lambda)L)^2+(b+(1-\lambda)M)^2,
	&u_{11}&=(a+L)^2+(b+M)^2.
\end{align*}

Set
\begin{equation*}
c_\lambda:=1-\lambda-\lambda^2,
\qquad
d_\lambda:=1-2\lambda.
\end{equation*}
Then \(c_\lambda,d_\lambda>0\) for
\(0<\lambda<1/2\).
For a cylinder pair $(I,J)$, we introduce the five quantities
\begin{equation}\label{eq:defOABCD}
\begin{aligned}
O(I,J)&:=bM-(a+L)L,\\
A(I,J)&:=\lambda bM-c_\lambda(a+L)L,\\
B(I,J)&:=\lambda aL-d_\lambda(b+M)M,\\
C(I,J)&:=1-a^2-b^2,\\
D(I,J)&:=(a+L)^2+(b+M)^2-1.
\end{aligned}
\end{equation}
The roles of these quantities are as follows. The condition
\(O\geq0\) fixes the order of the four first-generation image
intervals, while \(A\geq0\) and \(B\geq0\) guarantee their overlaps. Since $g$ is coordinatewise increasing on $[0,1]^2$, the conditions $C(I,J)>0$ and $D(I,J)>0$ are equivalent to
\begin{equation}\label{eq:1intgIJ}
1\in\interior g(I\times J)
=\bigl(a^2+b^2,(a+L)^2+(b+M)^2\bigr).
\end{equation}

Write \(I=I_u\) and \(J=I_v\), where $u\in \{0,1\}^n$ and  $v\in \{0,1\}^m$. For \(q\in \N_0\), we say that
\((I',J')\) is a \emph{relative level-\(q\) descendant pair} of \((I,J)\)
if $I'=I_{u\alpha}$ and $J'=I_{v\beta}$
for some \(\alpha,\beta\in\{0,1\}^q\). The associated
rectangle \(I'\times J'\) will be called a \emph{relative level-\(q\) descendant rectangle} of \(I\times J\).

Let \(\mathcal G(I,J)\) be the set of all
\(z\in g(I\times J)\) for which, for some \(q\geq1\), there exist
two distinct relative level-\(q\) descendant rectangles
\(R_1,R_2\) of \(I\times J\) such that
\[
z\in\interior g(R_1)
\cap\interior g(R_2).
\]
Whenever \((I',J')\) is a relative level-\(q\)
descendant pair of \((I,J)\), one has
\begin{equation}\label{eq:GIJinc}
	\mathcal G(I',J')\subset\mathcal G(I,J).
\end{equation}
Indeed, a relative level-\(k\) descendant pair of
\((I',J')\) is a relative level-\((q+k)\)
descendant pair of \((I,J)\). 

We first show that the conditions $O,A,B\geq0$ are preserved
when both cylinder intervals are replaced by descendants
of the same relative depth.

\begin{lem}\label{lem:weighted-inheritance}
Suppose that $I'=[a',a'+L']\subset I$ and $J'=[b',b'+M']\subset J$ form a  relative level-\(q\) descendant  pair, so that $
L'=\lambda^qL$,
$M'=\lambda^qM$.
If $O(I,J),A(I,J),B(I,J)\geq0$, then
$O(I',J'),A(I',J'),B(I',J')\geq0$.
\end{lem}

\begin{proof}
Let $s=\lambda^q$. Using
\[
a\leq a',\quad a'+L'\leq a+L,
\qquad
b\leq b',\quad b'+M'\leq b+M,
\]
we obtain
\begin{align*}
O(I',J')
&=s\bigl[b'M-(a'+L')L\bigr]
\geq sO(I,J),\\
A(I',J')
&=s\bigl[\lambda b'M-c_\lambda(a'+L')L\bigr]
\geq sA(I,J),\\
B(I',J')
&=s\bigl[\lambda a'L-d_\lambda(b'+M')M\bigr]
\geq sB(I,J).
\end{align*}
The conclusion follows.
\end{proof}

\begin{lem}\label{lem:one-step-double}
If $O(I,J), B(I,J)\geq0$, then the four intervals $G_{00},G_{10},G_{01},G_{11}$ cover $g(I\times J)$ without gaps, and
$(\ell_{10},u_{01})\subset\mathcal{G}(I,J)$.
\end{lem}

\begin{proof}
The condition $O(I,J)\geq0$ gives $bM\geq(a+L)L$. Hence
\begin{align*}
\ell_{01}-\ell_{10}
&=(1-\lambda)\bigl[2bM+(1-\lambda)M^2-2aL-(1-\lambda)L^2\bigr]\\
&\geq(1-\lambda)\bigl[(1+\lambda)L^2+(1-\lambda)M^2\bigr]>0,
\end{align*}
and similarly
\begin{align*}
u_{01}-u_{10}
&=(1-\lambda)\bigl[2(bM-aL)+(1+\lambda)(M^2-L^2)\bigr]\\
&\geq(1-\lambda)\bigl[(1-\lambda)L^2+(1+\lambda)M^2\bigr]>0.
\end{align*}
It follows that $\ell_{00}<\ell_{10}<\ell_{01}<\ell_{11}$ and $u_{00}<u_{10}<u_{01}<u_{11}$.

From $B(I,J)\geq0$,
\begin{align*}
u_{00}-\ell_{01}
&=2\lambda aL-2d_\lambda bM+\lambda^2L^2-d_\lambda M^2\\
&\geq\lambda^2L^2+d_\lambda M^2>0,
\end{align*}
and
\[
u_{10}-\ell_{11}=(u_{00}-\ell_{01})+2\lambda(1-\lambda)L^2>0.
\]
Therefore, the intervals
\(G_{00},G_{10},G_{01},G_{11}\), in this order, form an
overlapping chain from $\ell_{00}$ to $u_{11}$, and hence
\[
G_{00}\cup G_{10}\cup G_{01}\cup G_{11}
=
[\ell_{00},u_{11}]
=
g(I\times J).
\] 
Moreover,
\begin{align*}
(\ell_{10},u_{00})&=\interior G_{00}\cap\interior G_{10},\\
(\ell_{01},u_{10})&=\interior G_{10}\cap\interior G_{01},\\ 
(\ell_{11},u_{01})&=\interior G_{01}\cap \interior G_{11}.
\end{align*}
Since $u_{00}>\ell_{01}$ and $u_{10}>\ell_{11}$, these three open intervals overlap successively and their union is $(\ell_{10},u_{01})$. Therefore $(\ell_{10},u_{01})\subset\mathcal{G}(I,J)$.
\end{proof}

For $q\geq 0$, put $t_q:=\lambda^q$. The relative level-\(q\) descendants which keep the left endpoints fixed are
\[
I_q^-=[a,a+t_qL],
\qquad
J_q^-=[b,b+t_qM].
\]
The central double-covering interval obtained from Lemma~\ref{lem:one-step-double} is $(\alpha_q,\beta_q)$, where
\begin{align*}
\alpha_q:=\bigl(a+(1-\lambda)t_qL\bigr)^2+b^2,\quad \beta_q:=\bigl(a+\lambda t_qL\bigr)^2+\bigl(b+t_qM\bigr)^2.
\end{align*}
The intervals \((\alpha_q,\beta_q)\) approach the left endpoint
of the parent image. The next lemma shows that they overlap
successively and fill $(a^2+b^2,\beta_0)$.

\begin{lem}\label{lem:left-chain}
If $O(I,J),A(I,J),B(I,J)\geq0$, then
$(a^2+b^2,\beta_0)\subset\mathcal{G}(I,J)$.
\end{lem}

\begin{proof}
	By Lemmas~\ref{lem:weighted-inheritance}
	and~\ref{lem:one-step-double}, together with \eqref{eq:GIJinc},
	we have $(\alpha_q,\beta_q)\subset\mathcal G(I,J)$
	for every $q\in\N_0$.
	Since $A(I,J)\geq0$,
	we have $\lambda bM-c_\lambda aL\geq c_\lambda L^2$.
	Using $0<t_q\leq1$, we obtain
\begin{align*}		\beta_{q+1}-\alpha_q
			&=2t_q(\lambda bM-c_\lambda aL)
			+t_q^2\bigl[(\lambda^4-(1-\lambda)^2)L^2+\lambda^2M^2\bigr]\\
			&\geq t_q\bigl[2c_\lambda-t_q(1-\lambda)^2\bigr]L^2\\
			&\geq t_q(1-3\lambda^2)L^2>0.
		\end{align*}
	Thus consecutive intervals $(\alpha_q,\beta_q)$ overlap.
	Since $\alpha_q,\beta_q\downarrow a^2+b^2$, the union
	of these intervals equals $(a^2+b^2,\beta_0)$, proving the claim.
\end{proof}

The relative level-\(q\) descendants which keep the right endpoints fixed are
\begin{equation*}
I_q^+=[a_q,a+L],
\qquad
J_q^+=[b_q,b+M],
\end{equation*}
where \(a_q:=a+L-t_qL\) and \(b_q:=b+M-t_qM\). 
Set
\begin{equation*}
s_q:=a_q^2+b_q^2,
\qquad
\gamma_q:=\bigl(a_q+\lambda t_qL\bigr)^2+\bigl(b+M\bigr)^2.
\end{equation*}
The next lemma shows that the intervals $(s_q,\gamma_q)$
overlap successively and fill the interior of $g(I\times J)$.

\begin{lem}\label{lem:right-chain}
If $O(I,J),A(I,J),B(I,J)\geq0$, then
\[
\bigl(a^2+b^2,(a+L)^2+(b+M)^2\bigr)
\subset\mathcal{G}(I,J).
\]
\end{lem}

\begin{proof}
	By Lemmas~\ref{lem:weighted-inheritance}
	and~\ref{lem:left-chain}, together with \eqref{eq:GIJinc},
	we have $(s_q,\gamma_q)\subset\mathcal G(I,J)$
	for every $q\in\N_0$.
	Using $c_\lambda-d_\lambda=\lambda(1-\lambda)$,
	a direct expansion gives
		\begin{align*}
			\gamma_q-s_{q+1}
			&=2t_q\bigl[\lambda(b+M)M-d_\lambda(a+L)L\bigr]
			+t_q^2(d_\lambda L^2-\lambda^2M^2)\\
			&=2t_qA(I,J)+2\lambda(1-\lambda)t_q(a+L)L\\
			&\qquad+d_\lambda t_q^2L^2
			+\lambda t_q(2-\lambda t_q)M^2>0.
		\end{align*}
	Here the last inequality follows from $A(I,J)\geq0$,
	$d_\lambda>0$, and $0<\lambda t_q<1/2$.
	Thus the intervals $(s_q,\gamma_q)$ overlap successively.
	Since $s_q,\gamma_q\uparrow(a+L)^2+(b+M)^2$
	and $s_0=a^2+b^2$, their union equals
	$\bigl(a^2+b^2,(a+L)^2+(b+M)^2\bigr)$.
	This proves the lemma.
\end{proof}

\begin{lem}\label{lem:Gfull}
If $O(I,J),A(I,J),B(I,J)\geq0$, then
$\mathcal{G}(I,J)=\interior g(I\times J)$.
\end{lem}

\begin{proof} 
The inclusion ``$\supseteq$'' follows from Lemma~\ref{lem:right-chain}.
For every descendant rectangle $R\subset I\times J$,
we have $\interior g(R)\subseteq\interior g(I\times J)$.
The reverse inclusion therefore follows from the
definition of $\mathcal G(I,J)$.
\end{proof}

We can now convert interval double covering into continuum many
circle points by iterating the construction.

\begin{prop}\label{prop:continuum-criterion}
Suppose that a cylinder pair $(I,J)$ satisfies
\begin{equation}\label{eq:five-criterion}
O(I,J),A(I,J),B(I,J)\geq0,
\qquad
C(I,J),D(I,J)>0.
\end{equation}
Then $\#E_\lambda=\cardc$.
\end{prop}

\begin{proof}
By \eqref{eq:1intgIJ},  $1\in \interior g(I\times J)$. By Lemma~\ref{lem:Gfull}, there exist an integer $q\geq 1$ and two distinct relative level-\(q\)  descendant rectangles of $I\times J$ whose $g$-images both contain $1$ in their interiors. Lemma~\ref{lem:weighted-inheritance} shows the corresponding descendant pairs continue to satisfy $O,A,B\geq0$. Moreover, by
\eqref{eq:1intgIJ}, the fact that \(1\) belongs to the
interior of each descendant image is equivalent to the positivity
of the corresponding quantities \(C\) and \(D\). Thus each of the
two descendant pairs again satisfies \eqref{eq:five-criterion}.

We now iterate this construction. Set $(\mathcal I_\varnothing,\mathcal J_\varnothing):=(I,J)$.
Suppose that a cylinder pair
\((\mathcal I_\sigma,\mathcal J_\sigma)\) has been constructed for
some \(\sigma\in\{0,1\}^*\) and satisfies
\eqref{eq:five-criterion}. Applying the preceding argument to this
pair, choose an integer \(q_\sigma\geq1\) and two distinct relative
level-\(q_\sigma\) descendant pairs
$(\mathcal I_{\sigma0},\mathcal J_{\sigma0})$ and 
$(\mathcal I_{\sigma1},\mathcal J_{\sigma1})$,
such that
\[
1\in
\interior g(\mathcal I_{\sigma i}\times\mathcal J_{\sigma i}),
\qquad i\in\{0,1\}.
\]
Both child pairs again satisfy \eqref{eq:five-criterion}. Continuing
recursively produces a full binary tree of cylinder pairs indexed by
\(\{0,1\}^*\).

Let $\omega=(\omega_1,\omega_2,\ldots)\in\{0,1\}^{\N}$,
and write
\[
\omega|_k:=\omega_1\cdots\omega_k,
\qquad
\omega|_0:=\varnothing.
\]
The rectangles $\mathcal I_{\omega|_k}\times\mathcal J_{\omega|_k}$, $k\in\N_0$,
form a nested sequence with  diameters tending to zero. Hence their intersection
consists of a unique point, which we denote by $(x_\omega,y_\omega)$. The coding description of \(K_{\lambda}\) gives  $(x_\omega,y_\omega)\in K_{\lambda}\times K_{\lambda}$.

For each \(k\in\N_0\), we have
$1\in g(\mathcal I_{\omega|_k}\times\mathcal J_{\omega|_k})$.
Choose
$(x_k,y_k)\in
\mathcal I_{\omega|_k}\times\mathcal J_{\omega|_k}$
such that \(g(x_k,y_k)=1\). Since the diameters of the nested rectangles
tend to zero,
$(x_k,y_k)\longrightarrow(x_\omega,y_\omega)$.
By continuity of \(g\),
\[
x_\omega^2+y_\omega^2
=g(x_\omega,y_\omega)
=\lim_{k\to\infty}g(x_k,y_k)
=1.
\]
Thus $(x_\omega,y_\omega)\in E_\lambda$.

It remains to verify that distinct infinite binary sequences give
distinct points. Let \(\omega,\omega'\in\{0,1\}^{\N}\) be distinct,
and let \(r\geq1\) be the first index for which
\(\omega_r\neq\omega'_r\). The rectangles corresponding to
\(\omega|_r\) and \(\omega'|_r\) are the two distinct children of
the same node. They are relative descendants of a common depth, so
in at least one coordinate they involve distinct cylinder intervals
of the same level. Since \(\lambda<1/2\), distinct cylinder
intervals of the same level are disjoint. Hence the two rectangles
are disjoint in at least one coordinate, and consequently $(x_\omega,y_\omega)
\neq
(x_{\omega'},y_{\omega'})$.

Therefore, the map
\[
\{0,1\}^{\N}\longrightarrow E_\lambda,
\qquad
\omega\longmapsto(x_\omega,y_\omega),
\]
is injective. Since $\#\{0,1\}^{\N}=\cardc$,
we obtain \(\#E_\lambda\geq\cardc\). On the other hand,
\(E_\lambda\subset\R^2\), so \(\#E_\lambda\leq\cardc\). Therefore, $\#E_\lambda=\cardc$.
\end{proof}

\subsection{Verification of the double-covering criterion}

To apply Proposition~\ref{prop:continuum-criterion}
throughout the range \(\lambda_*<\lambda<1/2\), it suffices to construct, for each such parameter, a cylinder
pair $(I,J)$ satisfying
\[
O(I,J),A(I,J),B(I,J)\geq0,
\qquad
C(I,J),D(I,J)>0.
\]
We first treat the simpler range $2/5\leq \lambda<1/2$ using a fixed pair of short words. We then handle the more delicate range
\(\lambda_*<\lambda\leq2/5\) by a parameter-dependent
construction near \((0,1)\). Finally, we show that the criterion
cannot hold at or below \(\lambda_*\).

\begin{prop}\label{prop:highrange}
	If $2/5\leq\lambda<1/2$, then, with $I:=I_{010}$ and 
	$J:=I_{1110}$, one has
	\[
	O(I,J),A(I,J),B(I,J),C(I,J),D(I,J)>0.
	\]
	Consequently,  $\#E_\lambda=\cardc$.  
\end{prop}

\begin{proof}
For this cylinder pair, we have  $I=[a,a+L]$, $J=[b,b+M]$, where 
	\begin{equation*}
		a=\lambda-\lambda^2,
		\quad L=\lambda^3,
		\qquad
		b=1-\lambda^3,
		\quad M=\lambda^4.
	\end{equation*}
Substituting these values into \eqref{eq:defOABCD}, a direct calculation gives
\[
\begin{aligned}
		O(I,J)&=\lambda^5(1-\lambda-\lambda^2),\\
		A(I,J)&=\lambda^4(3\lambda-\lambda^2-1),\\
		B(I,J)&=\lambda^4Q_B(\lambda),\\
		C(I,J)&=\lambda^2Q_C(\lambda),\\
		D(I,J)&=\lambda^2Q_D(\lambda),
\end{aligned}
\]
	where
	\begin{align*}
		Q_B(x)&:=2x^5-3x^4+x^3-x^2+3x-1,\\
		Q_C(x)&:=-x^4-x^2+4x-1,\\
		Q_D(x)&:=x^6-2x^5+2x^4-2x^3+5x^2-4x+1.
	\end{align*}
These identities immediately give \( O(I, J) > 0 \) and \( A(I, J) > 0 \). 
For \(x\in[2/5,1/2]\), we have
\[
\begin{aligned}
	Q_B'(x)&=10x^4+3x^2-12x^3-2x+3\geq 3-12\left(\frac{1}{2}\right)^3
	-2\left(\frac{1}{2}\right)
	=\frac{1}{2}>0,\\
	Q_C'(x)
	&=4-2x-4x^3
	\geq4-2\left(\frac{1}{2}\right)
	-4\left(\frac{1}{2}\right)^3
	=\frac{5}{2}>0.
\end{aligned}
\]
Hence \(Q_B\) and \(Q_C\) are strictly increasing on
\([2/5,1/2]\).
Since
\[
Q_B\left(\frac{2}{5}\right)=\frac{149}{3125}>0,
\qquad
Q_C\left(\frac{2}{5}\right)=\frac{259}{625}>0,
\]
both \(Q_B\) and \(Q_C\) are positive throughout
\([2/5,1/2]\). Note that
\[
Q_D(x) = (1 - 2x)^2 + x^2(1 - x)^2(1 + x^2) > 0.
\]
Hence all five quantities are strictly positive, and
Proposition~\ref{prop:continuum-criterion} gives
\(\#E_\lambda=\cardc\).
\end{proof}

We next treat the range $\lambda_*<\lambda\leq
\frac{2}{5}$, which requires a parameter-dependent asymptotic construction near the point \((0,1)\). For fixed \(\lambda\in (\lambda_*, 2/5]\), the integer \(t\) is chosen so that the leading terms in the five criterion quantities are positive, while taking \(p\) large moves the cylinder pair towards \((0,1)\) and makes the remaining error terms tend to zero.

Set
\begin{equation*}
\rho_\lambda:=\lambda(1+\lambda),
\qquad
\eta_\lambda:=2\rho_\lambda-1=2\lambda^2+2\lambda-1>0.
\end{equation*}
The quantity \(\eta_\lambda\) measures the margin above the
threshold  \(\lambda_*\); indeed,
\(\eta_\lambda>0\) if and only if
\(\lambda>\lambda_*\).
Let $t=t(\lambda)\in\N$ be the smallest integer satisfying
\begin{equation}\label{eq:t-choice}
\lambda^t\leq2\eta_\lambda,
\end{equation}
and write $\tau:=\lambda^t$.
Such an integer exists since \(\lambda^k\to0\) as \(k\to\infty\).
For \(\lambda\leq2/5\),
\[
\lambda-2\eta_\lambda
=2-3\lambda-4\lambda^2
\geq
2-\frac{6}{5}-\frac{16}{25}
=\frac{4}{25}>0.
\]
Hence $2\eta_\lambda<\lambda$.
Consequently, $t\geq 2$.  Minimality of $t$ also gives
$2\lambda\eta_\lambda<\tau\leq2\eta_\lambda$.

Let $\xi_q$ be the prefix of length $q$ of the alternating sequence $1010\cdots$, with $\xi_0=\varnothing$. For $p\in\N$, define
\begin{equation}\label{eq:uptvpt}
u_{p,t}:=0^p\xi_t,
\qquad
v_{p,t}:=1^{2p+1}0\xi_{t-2}.
\end{equation}
The construction is guided by the quadratic tangency of the unit
circle to the horizontal line $y=1$ at \((0,1)\):
\[
\sqrt{1-x^2}=1-\frac{x^2}{2}+O(x^4)  \qquad(x\to0).
\]
Thus a horizontal scale of order \(\lambda^p\) naturally corresponds
to a vertical displacement from $1$ of order \(\lambda^{2p}\). This explains
the initial blocks \(0^p\) and \(1^{2p+1}\) in the two words in \eqref{eq:uptvpt}.
The alternating tails provide a finer adjustment at the scale
\(\tau=\lambda^t\).
According to the parity of $t$, set
\begin{equation*}
\kappa:=
\begin{cases}
1,&t\ \text{even},\\
\lambda,&t\ \text{odd},
\end{cases}
\qquad
\vartheta:=1+\lambda-\kappa=
\begin{cases}
\lambda,&t\ \text{even},\\
1,&t\ \text{odd}.
\end{cases}
\end{equation*}
We further introduce the normalized quantities
\begin{equation*}
\delta:=\frac{\lambda^p}{1+\lambda},
\qquad
U:=\rho_\lambda+\kappa(1+\lambda)\tau,
\qquad
W:=-\rho_\lambda+(1+\lambda)\vartheta\tau.
\end{equation*}

We first record the endpoints and lengths of the alternating
cylinders in the normalized variables introduced above.

\begin{lem}\label{lem:endpoints}
Let $I=I_{u_{p,t}}=[a,a+L]$ and $J=I_{v_{p,t}}=[b,b+M]$.
Then
\begin{align}
a&=\delta(1-\kappa\tau),
&L&=(1+\lambda)\delta\tau,\label{eq:alt-a-L}\\
b&=1-\delta^2U,
&M&=(1+\lambda)^2\delta^2\tau,\label{eq:alt-b-M}\\
a+L&=\delta(1+\vartheta\tau),
&b+M&=1+\delta^2W.\label{eq:alt-right-endpoints}
\end{align}
\end{lem}

\begin{proof}
In the word $\xi_t$, the digit $1$ occurs in the odd positions. Thus
\[
(1-\lambda)
\sum_{\substack{1\leq j\leq t\\j\ \text{odd}}}
\lambda^{j-1}
=
\frac{1-\kappa\lambda^t}{1+\lambda}.
\]
Therefore,
\[
a
=
\lambda^p\frac{1-\kappa\lambda^t}{1+\lambda}
=
\delta(1-\kappa\tau),
\]
while $L=\lambda^{p+t}=(1+\lambda)\delta\tau$.  The same computation for $v_{p,t}$ gives
\[
b=1-\frac{\lambda^{2p+1}+\kappa\lambda^{2p+t}}{1+\lambda},
\qquad
M=\lambda^{2p+t},
\]
which is equivalent to \eqref{eq:alt-b-M}. Adding the lengths and using $\kappa+\vartheta=1+\lambda$ yields \eqref{eq:alt-right-endpoints}.
\end{proof}

To separate the dependence on \(p\) in the five criterion
quantities, define
\[
\begin{aligned}
	\Phi_O&:=\lambda-\vartheta\tau,\\
	\Phi_A&:=\eta_\lambda-c_\lambda\vartheta\tau,\\
	\Phi_B&:=\eta_\lambda-\lambda\kappa\tau,\\
	\Phi_C&:=\eta_\lambda
	+2\kappa(2+\lambda)\tau-\kappa^2\tau^2,\\
	\Phi_D&:=-\eta_\lambda
	+2\vartheta(2+\lambda)\tau+\vartheta^2\tau^2.
\end{aligned}
\]
The following identities express the five criterion quantities in
terms of these auxiliary quantities.

\begin{lem}\label{lem:OABCD-identities}
For the cylinders in Lemma~\ref{lem:endpoints},
\begin{equation}\label{eq:OABCDatIJ}
\begin{aligned}
O(I,J)&=(1+\lambda)\delta^2\tau
\bigl[\Phi_O-(1+\lambda)\delta^2U\bigr],\\
A(I,J)&=(1+\lambda)\delta^2\tau
\bigl[\Phi_A-\rho_\lambda\delta^2U\bigr],\\
B(I,J)&=(1+\lambda)\delta^2\tau
\bigl[\Phi_B-d_\lambda(1+\lambda)\delta^2W\bigr],\\
C(I,J)&=\delta^2\bigl[\Phi_C-\delta^2U^2\bigr],\\
D(I,J)&=\delta^2\bigl[\Phi_D+\delta^2W^2\bigr].
\end{aligned}
\end{equation}
\end{lem}

\begin{proof}
	Using the endpoint formulas in
	\eqref{eq:alt-a-L}--\eqref{eq:alt-right-endpoints}, we obtain
	\[
	\begin{aligned}
		\frac{O(I,J)}{(1+\lambda)\delta^2\tau}
		&=
		(1+\lambda)(1-\delta^2U)
		-(1+\vartheta\tau)=
		\Phi_O-(1+\lambda)\delta^2U,\\[1mm]
		\frac{A(I,J)}{(1+\lambda)\delta^2\tau}
		&=
		\lambda(1+\lambda)(1-\delta^2U)
		-c_\lambda(1+\vartheta\tau)=\Phi_A-\rho_\lambda\delta^2U,\\[1mm]
		\frac{B(I,J)}{(1+\lambda)\delta^2\tau}
		&=
		\lambda(1-\kappa\tau)
		-d_\lambda(1+\lambda)(1+\delta^2W)=
		\Phi_B-d_\lambda(1+\lambda)\delta^2W.
	\end{aligned}
	\]
	Similarly,
	\[
	\begin{aligned}
		\frac{C(I,J)}{\delta^2}
		&=
		2U-(1-\kappa\tau)^2-\delta^2U^2=\Phi_C-\delta^2U^2,\\[1mm]
		\frac{D(I,J)}{\delta^2}
		&=
		(1+\vartheta\tau)^2+2W+\delta^2W^2=\Phi_D+\delta^2W^2.
	\end{aligned}
	\]
	Here we have used $
	\rho_\lambda-c_\lambda=
	\lambda-d_\lambda(1+\lambda)=\eta_\lambda$,
	together with the definitions of \(U,W\) and
	\(\Phi_O,\ldots,\Phi_D\). This proves
	\eqref{eq:OABCDatIJ}.
\end{proof}

In view of Lemma~\ref{lem:OABCD-identities}, it remains to show
that the five auxiliary quantities above are strictly positive.

\begin{lem}\label{lem:PhiOABCD0}
Let $
\lambda_*<\lambda\leq2/5$ and let \(t,\tau,\kappa,\vartheta\) be defined as above. Then
\[
\Phi_O,\Phi_A,\Phi_B,\Phi_C,\Phi_D>0.
\]
\end{lem}

\begin{proof}
	We have $\lambda\leq\kappa,\vartheta\leq1$ and
	$2\lambda\eta_\lambda<\tau\leq2\eta_\lambda<\lambda$.
	Since $1-2c_\lambda=\eta_\lambda$, these bounds give
	\[
	\begin{aligned}
		\Phi_O&\geq\lambda-2\eta_\lambda>0,\\
		\Phi_A&\geq\eta_\lambda-2c_\lambda\eta_\lambda=\eta_\lambda^2>0,\\
		\Phi_B&\geq\eta_\lambda(1-2\lambda)>0.
	\end{aligned}
	\]
	Also, $0<\kappa\tau<1$ implies
	$\Phi_C=\eta_\lambda+\kappa\tau[2(2+\lambda)-\kappa\tau]>0$.
	Finally,
\[
\Phi_D>\eta_\lambda[-1+4\lambda\vartheta(2+\lambda)]
		\geq\eta_\lambda[-1+4\lambda^2(2+\lambda)]>0,
\]
	since $x\mapsto4x^2(2+x)$ is increasing for $x>0$ and
	$4\lambda_*^2(2+\lambda_*)=2(1-\lambda_*)>1$.
\end{proof}

The preceding two lemmas show that, by taking \(p\) sufficiently
large, the five criterion inequalities can be satisfied simultaneously.

\begin{prop}\label{prop:lowrange}
Let $
\lambda_*<\lambda\leq 2/5$ and choose \(t=t(\lambda)\) as in \eqref{eq:t-choice}. Then, for every sufficiently large $p\in\N$, the cylinder pair $(I,J):=(I_{u_{p,t}}, I_{v_{p,t}})$ satisfies
\begin{equation}\label{eq:OABCD-positive}
O(I,J),A(I,J),B(I,J),C(I,J),D(I,J)>0.
\end{equation}
\end{prop}

\begin{proof}
For fixed \(\lambda\) and \(t\), the quantities
\(U,W,\Phi_O,\ldots,\Phi_D\) are independent of \(p\), whereas
\[
\delta=\frac{\lambda^p}{1+\lambda}\longrightarrow0 \qquad\text{as }p\to\infty.
\]
By Lemma~\ref{lem:PhiOABCD0}, \(\Phi_O,\Phi_A,\Phi_B,\Phi_C,\Phi_D\) are strictly positive. Therefore, the identities in
\eqref{eq:OABCDatIJ} imply that \eqref{eq:OABCD-positive} holds for all sufficiently
large \(p\).
\end{proof}

The two constructions above cover the full parameter range
\((\lambda_*,1/2)\). We can now complete the proof of
Theorem~\ref{thm:intersection}(ii).

\begin{proof}[Proof of
	Theorem~\ref{thm:intersection}(ii)]
For \(2/5\leq\lambda<1/2\), Proposition~\ref{prop:highrange}
provides a cylinder pair satisfying \eqref{eq:five-criterion}.
For \(\lambda_*<\lambda\leq2/5\), such a pair is supplied by
Proposition~\ref{prop:lowrange}. In either case,
Proposition~\ref{prop:continuum-criterion} yields \(\#E_\lambda=\cardc\).
\end{proof}

We conclude the section by showing that the lower endpoint
\(\lambda_*\) cannot be improved within the criterion of
Proposition~\ref{prop:continuum-criterion}. However, this does not assert that
\(\lambda_*\) is the true continuum-cardinality threshold for
\(E_\lambda\).

\begin{prop}\label{prop:ABlambda-star}
	If a cylinder pair $(I,J)$ satisfies $A(I,J)\geq0$ and $B(I,J)\geq0$, then
	\[
	\lambda>\lambda_*=\frac{\sqrt{3}-1}{2}.
	\]
\end{prop}

\begin{proof}
 Since \(I,J\subset[0,1]\), we have \(a,b\geq0\).
We first show that \(a,b>0\). If \(a=0\), then
\[B(I,J)=-d_\lambda(b+M)M<0,\]
contrary to the assumption \(B(I,J)\geq0\). Similarly, if \(b=0\),
then \[A(I,J)=-c_\lambda(a+L)L<0,\]
contrary to \(A(I,J)\geq0\).

Now the inequalities \(A(I,J)\geq0\) and \(B(I,J)\geq0\) give
\[
\lambda bM\geq c_\lambda(a+L)L,
\qquad
\lambda aL\geq d_\lambda(b+M)M.
\]
Multiplying these inequalities and cancelling $LM>0$ gives
	\[
	\lambda^2ab\geq c_\lambda d_\lambda(a+L)(b+M)>c_\lambda d_\lambda ab.
	\]
	Thus $\lambda^2>c_\lambda d_\lambda$. Since $\lambda^2-c_\lambda d_\lambda
	=(1-\lambda)(2\lambda^2+2\lambda-1)$, we obtain $\lambda>\lambda_*$. 
\end{proof}

\section{Hausdorff dimension of the circle intersection}\label{sec:dimension}

The double-covering argument in Section~\ref{sec:continuum}
establishes continuum cardinality, but it does not show that
\(E_\lambda\) has positive Hausdorff dimension. Since \(E_\lambda\)
is the graph of \(T\) over
$X_\lambda=K_\lambda\cap T(K_\lambda)$,
we reduce the dimension problem to this one-dimensional intersection
and study it using Newhouse thickness. The Hunt--Kan--Yorke theorem
gives positive Hausdorff dimension for
\(\lambda>\lambda_{\mathrm H}=\sqrt{2}-1\), while the quantitative
theorem of Falconer and Yavicoli yields explicit lower bounds near
\(\lambda=1/2\). In particular, these bounds imply
\(\dimH E_\lambda\to1\) as \(\lambda\uparrow1/2\), and a refinement
gives a two-sided first-order estimate for \(1-\dimH E_\lambda\).

\subsection{One-dimensional reduction and a general upper bound}

We record the following standard dimension reduction and an immediate upper bound for later use.
 
\begin{lem}\label{lem:dimension-reduction}
For $0<\lambda<1/2$,
\[
\dimH E_\lambda=\dimH X_\lambda.
\]
In particular, $\dimH E_\lambda\leq \dimH K_{\lambda}=\frac{\log2}{-\log\lambda}$.
\end{lem}

\begin{proof}
Projection onto the first coordinate is Lipschitz, so
$\dimH X_\lambda\leq\dimH E_\lambda$. Conversely, $\Gamma(x)=(x,T(x))$ is Lipschitz on $[0,1-1/m]$ for $m\geq 2$, and 
\[
E_\lambda=\{(1,0)\}\cup\bigcup_{m=2}^{\infty}\Gamma(X_{\lambda}\cap [0,1-1/m]).
\]
Countable stability of Hausdorff dimension gives the reverse inequality. Since \(X_\lambda\subset K_{\lambda}\) and $\dimH K_{\lambda}=\frac{\log2}{-\log\lambda}$,
the upper bound follows.
\end{proof}

\subsection{Local thickness and positive dimension}

We first recall the definition of Newhouse thickness of Cantor sets on $\R$. By a Cantor set on $\R$ we mean a nonempty, perfect, totally disconnected compact subset of $\R$.  Let $F\subset\R$ be a Cantor set and let $(G_n)$ be an enumeration of the connected components of $\conv(F)\setminus F$. We call  $(G_n)$ a \emph{presentation} of $F$. Let $\widetilde{G}_n$ be the component of $\conv(F)\setminus\bigcup_{j<n}G_j$ that contains $G_n$, and let \(L_n\) and \(R_n\) be the two
closed components of \(\widetilde{G}_n\setminus G_n\). The thickness of $(G_n)$ is
\[
\tau((G_n)):=\inf_n\frac{\min\{|L_n|,|R_n|\}}{|G_n|}.
\]
Then the \emph{Newhouse thickness} of $F$, denoted by $\tau_N(F)$, is defined to be the supremum of $\tau((G_n))$  over all presentations. It is standard that this supremum is attained by a presentation in
which the gaps $G_n$ are ordered by nonincreasing length; see
\cite[Section~II.1]{Moreira1996}. Consequently,  this definition
agrees with the decreasing-gap definition used in
\cite[Definition~1]{FalconerYavicoli2022}. It is immediate from the definition that Newhouse thickness is invariant under similarities. Moreover, the standard decreasing-gap presentation of the
central Cantor set \(K_{\lambda}\) gives
\[
\tau_N(K_{\lambda})=\frac{\lambda}{1-2\lambda}.
\]

We will use the following elementary distortion estimate for
Newhouse thickness under \(C^1\) diffeomorphisms.

\begin{lem}\label{lem:tauN-distortion}
Let \(J\subset\R\) be an open interval, let $F\subset J$ be a Cantor set, and let $h:J\to h(J)$ be a monotone $C^1$ diffeomorphism. Suppose that 
\[
0<m_h\leq |h'(x)|\leq M_h<\infty
\qquad(x\in \conv F).
\]
Then
\begin{equation*}
\tau_N(h(F))\geq\frac{m_h}{M_h}\tau_N(F).
\end{equation*}
\end{lem}

\begin{proof}
	Fix a presentation \(\mathcal P=(G_n)\) of \(F\). Since \(h\) is a
	monotone homeomorphism,
	\[
	\mathcal P_h:=(h(G_n))
	\]
	is a presentation of \(h(F)\). Moreover, \(h\) maps the parent
	interval and the two bridges associated with \(G_n\) onto those
	associated with \(h(G_n)\), possibly reversing their left--right
	order.
	
	By the mean value theorem,
	\[
	|h(L_n)|\geq m_h|L_n|,
	\qquad
	|h(R_n)|\geq m_h|R_n|,
	\qquad
	|h(G_n)|\leq M_h|G_n|.
	\]
	Consequently,
	\[
	\frac{\min\{|h(L_n)|,|h(R_n)|\}}{|h(G_n)|}
	\geq
	\frac{m_h}{M_h}
	\frac{\min\{|L_n|,|R_n|\}}{|G_n|}.
	\]
	Taking the infimum over \(n\) and then the supremum over presentations proves the claim.
\end{proof}

Two Cantor sets $F_1,F_2\subset\R$ are called
\emph{interleaved} if each set meets the interior of the
convex hull of the other, that is,
\[
F_1\cap\interior(\conv F_2)\neq\varnothing,
\qquad
F_2\cap\interior(\conv F_1)\neq\varnothing;
\]
see \cite[Definition~2.5]{McDonaldTaylor2024}.

We will use the following theorem of Hunt, Kan, and Yorke.

\begin{thm}[Hunt--Kan--Yorke {\cite[Theorem~1]{HuntKanYorke1993}}]\label{thm:HKY}
Let $F_1,F_2$ be interleaved Cantor sets, and write their
 Newhouse thicknesses in decreasing order as $u\geq v>0$. If
\begin{equation}\label{eq:HKY-conditions}
u>1+\frac{3}{v}+\frac{1}{v^2},
\qquad
v>\frac{4}{u}+\frac{4}{u^2}+\frac{1}{u^3},
\end{equation}
then $F_1\cap F_2$ contains a Cantor set of positive Newhouse thickness.
\end{thm}

In the symmetric case \(u=v\),  the two inequalities in
\eqref{eq:HKY-conditions} are equivalent to
\[
(u+1)(u^2-2u-1)>0
\quad\text{and}\quad
(u+1)^2(u^2-2u-1)>0,
\]
respectively. Hence,  they hold simultaneously
if and only if \(u>1+\sqrt{2}\). 
Also note that 
\begin{equation}\label{eq:sqrt2m1}
\frac{\lambda}{1-2\lambda}>1+\sqrt{2}
\quad\Longleftrightarrow\quad
\lambda>\sqrt{2}-1=\lambda_{\mathrm H}.
\end{equation}
To convert the positive-thickness conclusion of Theorem~\ref{thm:HKY} into a Hausdorff-dimension statement, we use the standard estimate
\begin{equation}\label{eq:thickness-dimension}
\dimH F\geq\frac{\log2}{\log(2+1/\tau_N(F))}>0
\end{equation}
for every Cantor set \(F\) with $\tau_N(F)>0$; see, for example, \cite[p.~77]{PalisTakens1993}.

We now localize near a point $(x_0,y_0)\in E_\lambda$
whose coordinates are not cylinder endpoints.
The localized Cantor sets, after applying $T$ to the second,
remain interleaved, while the derivative distortion of $T$
on shrinking cylinder intervals containing $y_0$ tends to one.

\begin{lem}\label{lem:local-cylinders}
Let $(x_0,y_0)\in E_\lambda$, and suppose neither coordinate is a cylinder endpoint. Let $I_n,J_n$ be the unique level-$n$ cylinder intervals containing $x_0,y_0$, respectively.  Define
\[
C_n:=K_{\lambda}\cap I_n,
\qquad
D_n:=K_{\lambda}\cap J_n.
\]
Then:
\begin{enumerate}[label=\textup{(\roman*)}]
\item $C_n$ and $T(D_n)$ are interleaved;
\item $\tau_N(C_n)=\tau_N(D_n)=\lambda/(1-2\lambda)$;
\item for every $\theta\in(0,1)$ and all sufficiently large $n$,
\begin{equation*}
\tau_N(T(D_n))\geq\theta\frac{\lambda}{1-2\lambda}.
\end{equation*}
\end{enumerate}
\end{lem}

\begin{proof}
Since $x_0,y_0$ are not cylinder endpoints, $x_0\in \interior I_n$ and $y_0\in \interior J_n$. Note that  $\conv C_n=I_n$ and $\conv D_n=J_n$. Moreover,
since  \(T\) is continuous and strictly decreasing on \([0,1]\), $\conv T(D_n)=T(J_n)$.
Hence \(x_0=T(y_0)\) belongs to $C_n\cap T(D_n)$ and to the interiors of their convex hulls, proving (i). Since $C_n$ and 
$D_n$ are similar copies of $K_{\lambda}$, (ii) follows by 
similarity invariance of Newhouse thickness.

Since \(y_0\in(0,1)\), we have $J_n\subset (0,1)$ for all sufficiently large $n$.  Since $\psi=|T'|$ 
is positive and continuous near $y_0$, while $\diam J_n\to0$,
\[
\frac{\min_{J_n}\psi}{\max_{J_n}\psi}\longrightarrow1.
\]
Lemma~\ref{lem:tauN-distortion}
now gives
$\tau_N(T(D_n))\geq(\min_{J_n}\psi/\max_{J_n}\psi)\tau_N(D_n)$,
which proves (iii).
\end{proof}

\begin{proof}[Proof of Theorem~\ref{thm:intersection}(iii)]
Since $\lambda>\lambda_{\mathrm H}>\lambda_*$, Theorem~\ref{thm:intersection}(ii) shows that $E_\lambda$ has the cardinality of the continuum. The set of cylinder endpoints is countable, and for each prescribed endpoint coordinate the first-quadrant unit circle determines at most one other coordinate. Hence there exists $(x_0,y_0)\in E_\lambda$ with neither coordinate a cylinder endpoint.

Set
\[
U:=\frac{\lambda}{1-2\lambda}.
\]
By \eqref{eq:sqrt2m1}, $U>1+\sqrt{2}$, so both inequalities in \eqref{eq:HKY-conditions} are satisfied at $(u,v)=(U,U)$. Choose $v_0\in(0,U)$ close enough to $U$ that
\[
U>1+\frac{3}{v_0}+\frac{1}{v_0^2},
\qquad
v_0>\frac{4}{U}+\frac{4}{U^2}+\frac{1}{U^3}.
\]
Put $\theta=v_0/U$. Choose sufficiently large $n$ that Lemma~\ref{lem:local-cylinders}(iii) holds and $I_n, J_n\subset (0,1)$. If $u\geq v$ are the thicknesses of  $C_n,T(D_n)$, then $u\geq U$ and $v\geq v_0$. The right-hand sides in \eqref{eq:HKY-conditions} decrease as the other thickness increases, so Theorem~\ref{thm:HKY} applies. Thus $C_n\cap T(D_n)$ contains a Cantor set $F$ of positive thickness. By \eqref{eq:thickness-dimension}, $\dimH F>0$. Since  $x\mapsto(x,T(x))$ is bi-Lipschitz on $I_n$, its image of $F$ is contained in $E_\lambda$, proving the claim.
\end{proof}

\begin{rem}
At \(\lambda=\lambda_{\mathrm H}=\sqrt{2}-1\), one has
\[
U=\frac{\lambda}{1-2\lambda}=1+\sqrt{2}.
\]
Thus the symmetric pair \((u,v)=(U,U)\) lies exactly on the
boundary of both inequalities in \eqref{eq:HKY-conditions}.
Consequently, the preceding argument does not yield positive Hausdorff
dimension at this endpoint.
\end{rem}

\subsection{A quantitative lower bound near \texorpdfstring{$1/2$}{1/2}}

We use the two-set, one-dimensional specialization of the quantitative
intersection theorem of Falconer and Yavicoli
\cite[Theorem~6]{FalconerYavicoli2022}. For $d=1$, their  thickness agrees with Newhouse thickness, and the constants in
\cite[Definition~5]{FalconerYavicoli2022} become $K_1=96\log 16$ and $K_2=186624$.
For the finite family $\{F_1,F_2\}$ of Cantor sets in $\R$, the uniform boundedness condition
on the diameters is automatic, and every nondegenerate closed
interval $
B\subset\conv F_1\cap\conv F_2$
is a closed ball disjoint from the unbounded components of
$\R\setminus F_1$ and $\R\setminus F_2$. Thus the geometric hypotheses of \cite[Theorem~6]{FalconerYavicoli2022} are satisfied, and the resulting specialization is as follows.

\begin{thm}[Falconer--Yavicoli {\cite[Theorem~6]{FalconerYavicoli2022}}]\label{thm:FY}
Let $F_1,F_2\subset\R$ be Cantor sets of Newhouse thicknesses $\tau_1,\tau_2>0$. Let $B$ be a nondegenerate closed interval satisfying
\[
B\subset\conv F_1\cap\conv F_2.
\]
Put
\[
\Delta:=\max\{\diam F_1,\diam F_2\},
\qquad
\beta:=\min\left\{\frac{1}{4},\frac{|B|}{\Delta}\right\}.
\]
If there exists $c\in(0,1)$ such that
\begin{equation*}
Q:=\tau_1^{-c}+\tau_2^{-c}
\leq\frac{\beta^c(1-\beta^{1-c})}{186624},
\end{equation*}
then
\begin{equation*}
\dimH(B\cap F_1\cap F_2)
\geq1-\frac{96\log16}{\beta|\log\beta|}Q^{1/c}.
\end{equation*}
\end{thm}

Consider the level-three cylinder $K_{101}=f_{101}(K_{\lambda})$, and write 
$\conv K_{101}=[a_\lambda,b_\lambda]$,
where
\begin{equation}\label{eq:a-b-lambda}
a_\lambda:=1-\lambda+\lambda^2-\lambda^3,
\qquad
b_\lambda:=1-\lambda+\lambda^2.
\end{equation}
The word $101$ is chosen because, as \(\lambda \uparrow 1/2\), its cylinder interval approaches \([5/8, 3/4]\). This interval is compactly contained in $(0,1)$ and has a nondegenerate
overlap with \(T([5/8, 3/4])\).

Recall that $\psi=|T'|$ on $(0,1)$. 
Define
\begin{equation}\label{eq:r-V}
r_\lambda:=\frac{\psi(a_\lambda)}{\psi(b_\lambda)},
\qquad
V_\lambda:=r_\lambda\frac{\lambda}{1-2\lambda}.
\end{equation}
Since $\psi$ is increasing on $(0,1)$, Lemma~\ref{lem:tauN-distortion} gives
\begin{equation*}
\tau_N(K_{101})=U_\lambda:=\frac{\lambda}{1-2\lambda},
\qquad
\tau_N(T(K_{101}))\geq V_\lambda.
\end{equation*}

The convex hulls of $K_{101}$ and $T(K_{101})$ are $[a_\lambda,b_\lambda]$ and $[T(b_\lambda),T(a_\lambda)]$, respectively. Set
\begin{equation*}
	\ell_\lambda:=
	\min\{b_\lambda,T(a_\lambda)\}
	-\max\{a_\lambda,T(b_\lambda)\}.
\end{equation*}
Whenever \(\ell_\lambda>0\), define the nondegenerate closed interval
\begin{equation*}
B_\lambda:=
[\max\{a_\lambda,T(b_\lambda)\},
\min\{b_\lambda,T(a_\lambda)\}].
\end{equation*}
For such \(\lambda\), one has
\[
B_\lambda
=
\conv K_{101}
\cap
\conv T(K_{101}),
\qquad
|B_\lambda|=\ell_\lambda.
\]
Moreover,
\[
\diam K_{101}=\lambda^3,
\qquad
\diam T(K_{101})
=T(a_\lambda)-T(b_\lambda).
\]
Put
\[\Delta_\lambda:=\max\bigl\{
\diam K_{101},
\diam T(K_{101})
\bigr\}=\max\{\lambda^3,T(a_\lambda)-T(b_\lambda)\},\]
and 
\[\beta_\lambda:=\min\left\{\frac{1}{4},\frac{\ell_\lambda}{\Delta_\lambda}\right\}, \qquad Q_{\lambda,c}:=U_\lambda^{-c}+V_\lambda^{-c}, \quad c\in (0,1).\]
If \(\tau_1=\tau_N(K_{101})\) and
\(\tau_2=\tau_N(T(K_{101}))\), then
\[
\tau_1^{-c}+\tau_2^{-c}
\leq U_\lambda^{-c}+V_\lambda^{-c}
=Q_{\lambda,c}.
\]
Consequently, if
\begin{equation}\label{eq:FY-sufficient}
Q_{\lambda,c}
\leq\frac{\beta_\lambda^c(1-\beta_\lambda^{1-c})}{186624},
\end{equation}
then the hypotheses of Theorem~\ref{thm:FY} are satisfied. Thus Theorem~\ref{thm:FY} yields
\begin{equation}\label{eq:dim-lower-general}
\dimH E_\lambda
\geq1-\frac{96\log16}{\beta_\lambda|\log\beta_\lambda|}
Q_{\lambda,c}^{1/c}.
\end{equation}
Here $B_\lambda\cap K_{101}\cap T(K_{101})\subset X_\lambda$, so Lemma~\ref{lem:dimension-reduction} applies.

To make \eqref{eq:FY-sufficient} and
\eqref{eq:dim-lower-general} explicit, we next obtain uniform
estimates for \(\beta_\lambda\) and \(r_\lambda\) on
\(0.49\leq\lambda<1/2\).

\begin{lem}\label{lem:nearonehalf}
If $0.49\leq\lambda<1/2$, then  \(\ell_\lambda>0\) and
\begin{equation*}
\beta_\lambda=\frac{1}{4},
\qquad
r_\lambda\geq\rho_*:=\frac{15}{4\sqrt{39}}.
\end{equation*}
\end{lem}

\begin{proof}
The identity
\[
b_\lambda=\frac{3}{4}+\left(\frac{1}{2}-\lambda\right)^2
\]
gives
\[
\frac{3}{4}\leq b_\lambda\leq\frac{7501}{10000}<\frac{19}{25}<\frac{4}{5}.
\]
Also $a_\lambda'=-1+2\lambda-3\lambda^2<0$, so
\[
\frac{5}{8}=a_{1/2}\leq a_\lambda\leq a_{0.49}=\frac{632451}{10^6}<\frac{16}{25}.
\]
It follows that
\[
a_\lambda^2+b_\lambda^2
<
\left(\frac{16}{25}\right)^2
+
\left(\frac{19}{25}\right)^2
=
\frac{617}{625}<1.
\]
Hence $T(b_\lambda)>a_\lambda$ and $T(a_\lambda)>b_\lambda$.
Since $b_\lambda\geq3/4>1/\sqrt{2}$, we also have $b_\lambda>T(b_\lambda)$.  Consequently,
\[
\ell_\lambda=b_\lambda-T(b_\lambda)>0, \qquad B_\lambda=[T(b_\lambda),b_\lambda].
\]
The function $x\mapsto x-T(x)$ is increasing on $(0,1)$. Therefore,
\[
\ell_\lambda\geq
\frac{3}{4}-T\left(\frac{3}{4}\right)=\frac{3-\sqrt{7}}{4}>\frac{1}{16}.
\]
On the other hand,
\[
\lambda^3<\frac{1}{8}<\frac{1}{4},
\qquad
T(a_\lambda)-T(b_\lambda)
\leq
T\left(\frac{5}{8}\right)-T\left(\frac{4}{5}\right)=\frac{\sqrt{39}}{8}-\frac{3}{5}<\frac{1}{4}.
\]
Hence $\Delta_\lambda<1/4$ and $\ell_\lambda/\Delta_\lambda>1/4$, proving $\beta_\lambda=1/4$.

Finally, since \(\psi\) is increasing, $a_\lambda\geq5/8$, and 
$b_\lambda<4/5$,
we obtain
\[
r_\lambda
\geq\frac{\psi(5/8)}{\psi(4/5)}
=\frac{5/\sqrt{39}}{4/3}
=\frac{15}{4\sqrt{39}}.
\]
\end{proof}

We next derive two explicit numerical estimates needed to verify
\eqref{eq:FY-sufficient} and to obtain the constant in the
dimension bound in Theorem~\ref{thm:intersection}(iv).

Set
\[
c_0:=\frac{25}{26},
\qquad
\theta_0:=\rho_*^{-c_0}
=\left(\frac{4\sqrt{39}}{15}\right)^{25/26},
\qquad
\sigma_0:=4^{c_0-1}=4^{-1/26}.
\]
A direct calculation gives
\begin{align}
	4\left[
	\frac{186624(1+\theta_0)}{1-\sigma_0}
	\right]^{26/25}
	&<7.1972\times 10^7<
	7.2\times 10^7,
	\label{eq:1theta01sigma0}\\
	768(1+\theta_0)^{26/25}
	&<2101.974<2102.
	\label{eq:768theta0}
\end{align}

\begin{proof}[Proof of Theorem~\ref{thm:intersection}(iv)]
Suppose that $1/2-3\times10^{-9}\leq\lambda<1/2$. 
By Lemma~\ref{lem:nearonehalf}, $\beta_\lambda=1/4$ and $V_\lambda\geq\rho_*U_\lambda$. Taking $c=c_0=25/26$, we obtain
\begin{equation}\label{eq:Qc0-bound}
Q_{\lambda,c_0}
\leq U_\lambda^{-c_0}(1+\theta_0).
\end{equation}
Since $\beta_\lambda=1/4$, condition \eqref{eq:FY-sufficient} becomes
\[
Q_{\lambda,c_0}
\leq\frac{4^{-c_0}(1-\sigma_0)}{186624}.
\]
In view of \eqref{eq:Qc0-bound}, it is enough that
\begin{equation}\label{eq:U-sufficient}
U_\lambda
\geq4\left[\frac{186624(1+\theta_0)}{1-\sigma_0}\right]^{1/c_0}.
\end{equation}
Since $U_{\lambda}$ is strictly increasing on \((0,1/2)\) and  $\lambda\geq1/2-3\times10^{-9}=499999997/10^9$,
\[
U_\lambda=\frac{\lambda}{1-2\lambda}
\geq\frac{499999997}{6}>
7.2\times10^7.
\]
Thus \eqref{eq:U-sufficient} follows from \eqref{eq:1theta01sigma0}.

Using \eqref{eq:dim-lower-general} and
\[
\frac{96\log16}{(1/4)|\log(1/4)|}=768,
\]
we obtain
\begin{align*}
\dimH E_\lambda
&\geq1-768Q_{\lambda,c_0}^{1/c_0}\\
&\geq1-768(1+\theta_0)^{26/25}U_\lambda^{-1}\\
&\geq1-2102U_\lambda^{-1},
\end{align*}
where the last step is \eqref{eq:768theta0}. Since $U_\lambda^{-1}=(1-2\lambda)/\lambda$, we conclude that
\[
\dimH E_\lambda
\geq
1-2102\frac{1-2\lambda}{\lambda}.
\]

Finally, since \(E_\lambda\subset S^1\), one has
\(\dimH E_\lambda\leq1\). Letting
\(\lambda\uparrow1/2\) in the preceding lower bound therefore gives $\lim_{\lambda\uparrow1/2}\dimH E_\lambda=1$. This completes the proof. 
\end{proof}

We conclude this section with a sharper two-sided first-order estimate.
 
\begin{thm}\label{thm:dim-asy}
As $\lambda\uparrow1/2$,
\begin{equation}\label{eq:dim2sided}
\frac{1}{\log2}
\leq\liminf_{\lambda\uparrow1/2}
\frac{1-\dimH E_\lambda}{1-2\lambda}
\leq\limsup_{\lambda\uparrow1/2}
\frac{1-\dimH E_\lambda}{1-2\lambda}
\leq1536\left(1+\frac{3\sqrt{39}}{5\sqrt{7}}\right).
\end{equation}
\end{thm}

\begin{proof}
Write $\varepsilon=1-2\lambda$, so $\lambda=(1-\varepsilon)/2$. Then \(\varepsilon\downarrow0\) as \(\lambda\uparrow1/2\). By Lemma~\ref{lem:dimension-reduction}, 
\[
\dimH E_\lambda\leq\dimH K_{\lambda}=\frac{\log2}{-\log\lambda}.
\]
Note that 
\[
\frac{\log2}{-\log\lambda}
=\frac{\log2}{\log2-\log(1-\varepsilon)}
=1-\frac{\varepsilon}{\log2}+O(\varepsilon^2).
\]
It follows that the lower bound for the liminf in \eqref{eq:dim2sided} holds.

For the upper bound on the limsup, fix $c\in(0,1)$. When $\lambda$ is sufficiently close to $1/2$, Lemma~\ref{lem:nearonehalf} gives $\beta_\lambda=1/4$ and \(r_\lambda\geq\rho_*>0\). Hence
\[
Q_{\lambda,c}
\leq U_\lambda^{-c}(1+\rho_*^{-c})\longrightarrow0 \qquad\text{as }\lambda\uparrow\frac{1}{2}.
\] 
Since the right-hand side of \eqref{eq:FY-sufficient} is then a
fixed positive constant, condition \eqref{eq:FY-sufficient} holds
for all \(\lambda\) sufficiently close to \(1/2\). Using \eqref{eq:dim-lower-general}, we obtain
\begin{equation}\label{eq:dimlower}
\dimH E_\lambda
\geq1-768\frac{1-2\lambda}{\lambda}
(1+r_\lambda^{-c})^{1/c}.
\end{equation}
By \eqref{eq:a-b-lambda} and \eqref{eq:r-V},
\[
a_\lambda\to\frac{5}{8},
\qquad
b_\lambda\to\frac{3}{4},
\qquad
r_\lambda\to r_*:=\frac{5\sqrt{7}}{3\sqrt{39}}.
\]
Rearranging \eqref{eq:dimlower}, dividing by
\(1-2\lambda\), and then letting \(\lambda\uparrow1/2\), we obtain
\[
\limsup_{\lambda\uparrow1/2}
\frac{1-\dimH E_\lambda}{1-2\lambda}
\leq
1536(1+r_*^{-c})^{1/c}.
\]
Finally, letting \(c\uparrow1\) gives
\[
1536(1+r_*^{-1})
=
1536\left(1+\frac{3\sqrt{39}}{5\sqrt{7}}\right),
\]
which proves the upper bound in
\eqref{eq:dim2sided}.
\end{proof}

\section{Interior of the Minkowski sum}\label{sec:sum}

We now turn from the intersection problem to the Minkowski sum \[A_{\lambda}= C_\lambda + S^1.\]
Our main tool for proving
Theorem~\ref{thm:sum} is the following direct consequence of Lemma~\ref{lem:filling}. 

\begin{lem}\label{lem:nonlinear-fill}
Let $1/4\leq\lambda<1/2$. Let $a,b\in \R$ and $L>0$. Let $\phi$ be a $C^1$ function on an open
neighborhood of $[a,a+L]$, and define
\[
F(x,y)=y+\phi(x).
\]
Suppose that 
\begin{equation}\label{eq:phiprime}
1-2\lambda\leq\phi'(x)
\leq \frac{\lambda}{1-2\lambda} \qquad (x\in[a,a+L]).
\end{equation}
Then
\[
F\left((a+LK_\lambda)\times(b+LK_\lambda)\right)
=
[F(a,b),F(a+L,b+L)].
\]
\end{lem}

\begin{proof}
	Define
	\[
	G(s,t)
	:=b+L(1-s)+\phi(a+Lt).
	\]
By the hypothesis on $\phi$, this defines a $C^1$
function on an open neighborhood of $[0,1]^2$. On $[0,1]^2$, we have
\[
G(s,t)=
F(a+Lt,b+L(1-s)).
\]
Moreover, $\partial_sG=-L<0$,
	$\partial_t G=L\phi'(a+Lt)>0$, 
	and
	\[
	-\frac{\partial_sG}{\partial_tG}
	=
	\frac{1}{\phi'(a+Lt)}.
	\]
	By \eqref{eq:phiprime},
	\[
	\frac{1-2\lambda}{\lambda}
	\leq
	-\frac{\partial_sG}{\partial_tG}
	\leq
	\frac{1}{1-2\lambda}.
	\]
	Lemma~\ref{lem:filling} therefore gives
	$G(K_\lambda\times K_\lambda)
	=G([0,1]^2)$.
	Since $1-K_\lambda=K_\lambda$,
	\[
	G(K_\lambda\times K_\lambda)
	=F\left((a+LK_\lambda)\times(b+LK_\lambda)\right).
	\]
	Moreover, $G$ is strictly decreasing in $s$ and
	strictly increasing in $t$, so
	\[
	G([0,1]^2)
	=
	[G(1,0),G(0,1)]
	=
	[F(a,b),F(a+L,b+L)].
	\]
This completes the proof.
\end{proof}

We now apply Lemma~\ref{lem:nonlinear-fill} to construct
intervals contained in vertical sections of $A_{\lambda}$, using the maps $(x,y)\mapsto y+T(\alpha-x)$.  For $\alpha$ near $\sqrt{5}/5$ and $(x,y)$ in a sufficiently
small Cantor cylinder product near $(0,0)$,
the derivative with respect to $x$ is uniformly close
to $1/2$. For $\lambda>1/4$, the lemma gives nondegenerate interval images. Then continuity in $\alpha$ gives an open rectangle in $A_\lambda$.

For a fixed $\alpha$, define
\begin{equation*}
H_\alpha(x,y):=y+T(\alpha-x) \qquad\text{whenever }0<\alpha-x<1.
\end{equation*}
If $H_\alpha(x,y)=t$, then
\begin{equation}\label{eq:sum-decomposition}
(\alpha,t)=(x,y)+(\alpha-x,T(\alpha-x)),
\end{equation}
and the second vector belongs to $S^1$. Hence $x,y\in K_{\lambda}$ implies $(\alpha,t)\in A_{\lambda}$.
Writing $H_\alpha(x,y)=y+\phi_\alpha(x)$ with $\phi_\alpha(x)=T(\alpha-x)$, one has
\begin{equation*}
\phi_\alpha'(x)
=\frac{\alpha-x}{\sqrt{1-(\alpha-x)^2}}
=\psi(\alpha-x).
\end{equation*}
Note that $\psi$ is continuous and strictly increasing, and $\psi(u_0)=1/2$, where $u_0:=\sqrt{5}/5$.

Now we are ready to prove Theorem~\ref{thm:sum}. 

\begin{proof}[Proof of Theorem~\ref{thm:sum}]
Since $\lambda>1/4$, the value $\psi(u_0)=1/2$ lies in $(1-2\lambda,\lambda/(1-2\lambda))$.
By continuity of $\psi$ at $u_0$, choose $\delta>0$ such that
$[u_0-\delta,u_0+\delta]\subset(0,1)$
and
\begin{equation*}
1-2\lambda<\psi(u)<\frac{\lambda}{1-2\lambda} \qquad (|u-u_0|\leq\delta).
\end{equation*}
Choose $n\in \N$ with $L:=\lambda^n<\delta/2$. If $|\alpha-u_0|<\delta/2$ and $x\in [0,L]$, then $|\alpha-x-u_0|<\delta$. 
Therefore Lemma~\ref{lem:nonlinear-fill}, applied to
\(H_\alpha\), gives
\begin{equation}\label{eq:HalphaLK}
	H_\alpha(LK_{\lambda}\times LK_{\lambda})
	=[T(\alpha),L+T(\alpha-L)].
\end{equation}
Since $T$ is strictly decreasing, $T(u_0)<L+T(u_0-L)$. 
Choose a nondegenerate closed interval $J\subset (T(u_0),L+T(u_0-L))$. Continuity of the two endpoints in $\alpha$ gives an open interval $I$ containing $u_0$, with $I\subset (u_0-\delta/2,u_0+\delta/2)$, such that $J\subset [T(\alpha),L+T(\alpha-L)]$ for every $\alpha\in I$. Since $LK_{\lambda}=K_{0^n}\subset K_{\lambda}$, \eqref{eq:HalphaLK} and \eqref{eq:sum-decomposition} imply that $\interior (I\times J)\subset A_{\lambda}$.  This is a nonempty open rectangle.
\end{proof}

We now combine Theorem~\ref{thm:sum}
with the dimension and measure results of Simon and Taylor~\cite{SimonTaylor2020,SimonTaylor2022} to prove Corollary~\ref{cor:classification}.

\begin{proof}[Proof of Corollary~\ref{cor:classification}] 
For $\lambda>1/4$, $\interior A_{\lambda}\neq \varnothing$ by 
Theorem~\ref{thm:sum}. 
For \(0<\lambda<1/4\),
\cite[Theorem~2.1(b)]{SimonTaylor2022} gives $\dimH A_{\lambda}
=
1+\log4/\log(1/\lambda)
<2,$
so the interior is empty. At $\lambda=1/4$, $A_{\lambda}$ has zero area by  \cite[Remark~2.8(ii)]{SimonTaylor2020} and
\cite[Remark~2.10(2)]{SimonTaylor2022}, again
giving empty interior.
\end{proof}

We conclude this section by extending the argument of
Theorem~\ref{thm:sum} to regular $C^1$ closed curves. We call a $C^1$ map $\gamma:\R/\Z\to\R^2$ \emph{regular} if $\gamma'(t)\neq 0$ for every $t\in \R/\Z$. The image of such a regular map is called a regular $C^1$ closed curve in $\R^2$. 

\begin{cor}\label{cor:closed-curve-sum}
	Let $\Gamma\subset\R^2$ be the image of a regular $C^1$ map
	$\gamma:\R/\Z\to\R^2$.
	Then
	\[
	\interior(C_\lambda+\Gamma)\neq\varnothing
	\qquad (1/4<\lambda<1/2).
	\]
\end{cor}

\begin{proof}
	Write $\gamma(t)=(X(t),Y(t))$.
	At a maximum point $t_0$ of the function
	$t\mapsto Y(t)+X(t)/2$, we have
	\[
	Y'(t_0)+\frac12X'(t_0)=0.
	\]
	Regularity implies that $X'(t_0)\neq0$.
Hence $\Gamma$ contains the graph of a $C^1$ function $g$
on an open interval about $u_0:=X(t_0)$, with
$g'(u_0)=-1/2$.
	The proof of Theorem~\ref{thm:sum} now applies with $T$
	replaced by $g$, since it only requires the derivative
	of the graph function to be sufficiently close to $-1/2$
	on a small neighborhood.
\end{proof}

\begin{rem}
In Corollary~\ref{cor:closed-curve-sum}, closedness is used only to guarantee a tangent of slope
	$-1/2$.
By reflection and interchange of the coordinates, the same
conclusion holds for any regular $C^1$ arc $\Gamma$ having
a tangent of slope $\pm1/2$ or $\pm2$ at an interior point.
\end{rem}

\section{Concluding remarks and open problems}\label{sec:open-problems}

Corollary~\ref{cor:classification} settles the interior question for
$C_\lambda+S^1$. Several gaps remain in the cardinality and dimension
results for $E_\lambda$:
\begin{enumerate}[label=\textup{(\arabic*)}]
	\item For $2-\sqrt{3}<\lambda\leq\lambda_\infty$, a complete classification into finite and infinite intersections is not known. 
	
	\item For $\lambda_\infty<\lambda\leq(\sqrt{3}-1)/2$, the intersection is infinite, but it is not known whether
	\(E_\lambda\) has the cardinality of the continuum for every parameter in
	this range.   Proposition~\ref{prop:ABlambda-star} only shows that the inequalities underlying the double-covering criterion in 
	Proposition~\ref{prop:continuum-criterion} cannot hold for any $\lambda\leq (\sqrt{3}-1)/2$.
	\item For $(\sqrt{3}-1)/2<\lambda\leq\sqrt{2}-1$, the intersection has the cardinality of the continuum, but it is not known whether its Hausdorff dimension must be positive. At $\lambda=\sqrt{2}-1$, the two Hunt--Kan--Yorke
	inequalities become equalities for the symmetric thickness pair
	$(1+\sqrt2,1+\sqrt2)$, so the
	thickness argument used here is critical.
	\item The first-order constants in \eqref{eq:dim2sided} are far apart. Determining the exact asymptotic behavior of $1-\dimH E_\lambda$ as $\lambda\uparrow1/2$ appears to require substantially sharper quantitative intersection estimates.
\end{enumerate}

\noindent\textbf{Acknowledgements.}
We are grateful to Prof. Kan Jiang for informing us of his joint work with Zhikun Xie and for sharing their
manuscript~\cite{JiangXie2026} with us. We acknowledge their priority for Theorem~\ref{thm:intersection}(i). J. Li was supported by the National Natural Science Foundation of China (Grant No.~12571009) and the Natural Science Foundation of Hunan Province, China (Grant No.~2026JJ40003). Z. Shen was supported by the National Natural Science Foundation of China (Grant No.~12301011).  Y. Wu was supported by the National Natural Science Foundation of China
(Grant No.~12301110). AI tools (ChatGPT and DeepSeek) were used during the preparation of the paper to assist with the elaboration of some proof details, checking calculations, and improving the exposition. All mathematical statements and proofs were verified by the authors. The authors take full responsibility for the paper.

\end{document}